\documentclass[12pt]{article}
\usepackage{mathrsfs}
\usepackage{fancyhdr,graphicx}
\usepackage{amsfonts}
\usepackage{amssymb}
\usepackage{amsthm}
\usepackage{newlfont}
\usepackage{amsmath}
\usepackage{indentfirst}
\newtheorem{thm}{Theorem}[section]
\newtheorem{Lemma}[thm]{Lemma}
\newtheorem{cor}[thm]{Corollary}

\usepackage{latexsym,bm}
\title{{4-Factor-criticality of vertex-transitive graphs\thanks{This
work is supported by NSFC (grant no. 11371180).}}}

\author{Wuyang Sun, Heping Zhang\footnote{Corresponding author.} \\
\small{School of Mathematics and Statistics, Lanzhou University,
Lanzhou, Gansu 730000, P. R. China}
\\\small{E-mail addresses: sunwy09@lzu.edu.cn, zhanghp@lzu.edu.cn}}

\date{}

\begin{document}
\pagestyle{plain}
\pagenumbering{arabic}
\maketitle
\begin{abstract}
A graph of order $n$ is {\em $p$-factor-critical}, where $p$ is an integer of the same parity as $n$, if the removal of any set of $p$ vertices results in a graph with a perfect matching. 1-factor-critical graphs and 2-factor-critical graphs are well-known
factor-critical graphs and bicritical graphs, respectively. It is known that if a connected vertex-transitive graph has odd order, then it is factor-critical, otherwise it is elementary bipartite or bicritical. In this paper, we show that a connected vertex-transitive non-bipartite graph of even order at least $6$ is $4$-factor-critical if and only if its degree is at least 5. This result implies that each connected non-bipartite Cayley graphs of even order and degree at least 5 is 2-extendable.
\end{abstract}

\section{Introduction}

Only finite and simple graphs are considered in this paper. A {\em matching} of a graph is a set of its mutually nonadjacent edges. A {\em perfect matching} of a graph is a matching covering all its vertices. A graph is called {\em factor-critical} if the removal of any
one of its vertices results in a graph with a perfect matching. A graph is called {\em bicritical} if the removal of any pair of its distinct vertices results in a graph with a perfect matching. The concepts of factor-critical and bicritical graphs were introduced by Gallai \cite{Gallai} and by Lov\'{a}sz \cite{lovasz}, respectively. In matching theory, factor-critical graphs and bicritical graphs are two basic bricks in matching structures of graphs \cite{M.D. Plummer}. Later on, the two concepts were generalized
to the concept of $p$-factor-critical graphs by Favaron \cite{Favarvon} and Yu \cite{Yu}, independently. A graph of order $n$ is said to be {\em $p$-factor-critical}, where $p$ is an integer of the same parity as $n$, if the removal of any $p$ vertices results in a graph with a perfect matching.

$q$-extendable graphs was proposed by Plummer \cite{M.D. Plummer} in 1980. A connected graph of even order $n$ is {\em $q$-extendable}, where $q$ is an integer with $0\leq q<n/2$, if it has a perfect matching and every matching of size $q$ is contained in one of its perfect matchings. Favaron \cite{O. Favaron} showed that for $q$ even, every connected non-bipartite $q$-extendable graph is $q$-factor-critical. In 1993 Yu \cite{Yu} introduced an analogous concept for graphs of odd order. A connected graph of odd order is {\em $q\frac{1}{2}$-extendable}, if the removal of any one of its vertices results in a $q$-extendable graph.

A graph $G$ is said to be {\em vertex-transitive} if for any two vertices $x$ and $y$ in $G$ there is an automorphism $\varphi$ of $G$ such that $y=\varphi(x)$. A graph with a perfect matching is {\em elementary} if the union of its all perfect matchings forms a connected subgraphs. In \cite{plummer}, there is a following classic result about the factor-criticality and bicriticality of vertex-transitive graphs.

\begin{thm}[\cite{plummer}]\label{vtg} Let $G$ be a connected vertex-transitive
graph of order $n$. Then\\
(\textsl{a}) $G$ is factor-critical if $n$ is odd;\\
(b) $G$ is either elementary bipartite or bicritical if $n$ is even.
\end{thm}

A question arises naturally: Does a vertex-transitive non-bipartite graph has larger $p$-factor-criticality?

In fact, the $q$-extendability  and $q\frac{1}{2}$-extendability of Cayley graphs, an important class of vertex-transitive graphs, have been investigated in literature. It was proved in papers \cite{Chan,Chen,Miklavi} that a connected Cayley graph of even order on an abelian group, a dihedral group or a generalized dihedral group is 2-extendable except for several circulant graphs of degree at most 4.  Miklavi\v{c} and \v{S}parl \cite{Miklavi} also showed that a connected Cayley graph on an abelian group of odd order $n\geq3$ either is a cycle or is $1\frac{1}{2}$-extendable. Chan et al. \cite{Chan} raised the problem of characterizing 2-extendable Cayley graphs.

In \cite{Zhang}, we showed that a connected vertex-transitive graph of odd order $n\geq3$ is 3-factor-critical if and only if it is not a cycle. This general result is stronger than $1\frac{1}{2}$-extendability of Cayley graphs. In this paper, we obtain the following main result which gives a simple characterization of
4-factor-critical vertex-transitive non-bipartite graphs.

\begin{thm}\label{main} Let $G$ be a connected vertex-transitive non-bipartite graph of degree $k$ and of even order at least $6$. Then $G$ is $4$-factor-critical if and only if $k\geq5$.
\end{thm}

By Theorem \ref{main}, we know that all connected non-bipartite Cayley graphs of even order and degree at least 5 is 2-extendable.

The necessity of Theorem \ref{main} is clear. Our main task is to
show the sufficiency of Theorem \ref{main} by contradiction. Suppose
that $G$ is a connected non-bipartite vertex-transitive graph $G$ of
even order at least 6 and of degree at least 5 but $G$ is not
4-factor-critical. By the $s$-restricted edge-connectivity of $G$,
we find that in most cases a suitable integer $s$ can be chosen such
that every $\lambda_{s}$-atom of $G$ is an imprimitive block. Then
we can deduce contradictions. Some preliminary results are presented
in Section 2 and some properties of $\lambda_{s}$-atoms of $G$ which
are used to show their imprimitivity are proved in Section 3.
Finally, we complete the proof of Theorem \ref{main} in Section 4.

\section{Preliminaries}

In this section, we introduce some notations and results needed in this paper.

Let $G=(V(G),E(G))$ be a graph. For $X\subseteq V(G)$, let
$\overline{X}$ = $V(G)\backslash X$. For $Y\subseteq\overline{X}$,
denote by $[X,Y]$ the set of edges of $G$ with one end in $X$ and
the other in $Y$. In particular, we denote $[X,\overline{X}]$ by
$\nabla(X)$ and $|\nabla(X)|$ by $d_{G}(X)$. Denote by $N_{G}(X)$
the set of vertices in $\overline{X}$ which are ends of some edges
in $\nabla(X)$. If $X=\{v\}$, then $X$ is usually written to $v$.
Vertices in $N_{G}(v)$ are called the neighbors of $v$. If no
confusion exists, the subscript $G$ are usually omitted. Denote by
$G[X]$ the subgraph induced by $X$ and denote by $G-X$ the subgraph
induced by $\overline{X}$. The set of edges in $G[X]$ is denoted by
$E(X)$. Denote by $c_{0}(G)$ the number of the components of $G$
with odd order. For a subgraph $H$ of $G$, we denote
$d_{G}(V(H_{i}))$ and $\nabla(V(H_{i}))$ by $d_{G}(H_{i})$ and
$\nabla(H_{i})$, respectively.

For a connected graph $G$, a subset $F\subseteq E(G)$ is said to be an {\em edge-cut} of $G$ if $G-F$ is disconnected, where $G-F$ is the graph with vertex-set $V(G)$ and edge-set $E(G)\backslash F$. The {\em edge-connectivity} of $G$ is the minimum cardinality over all the edge-cuts of $G$, denoted by $\lambda(G)$. A subset $X\subseteq V(G)$ is called a {\em vertex-cut} of $G$ if $G-X$ is disconnected. The {\em vertex-connectivity} of $G$ of order $n$, denoted by $\kappa(G)$, is $n-1$ if $G$ is the complete graph $K_{n}$ and is the minimum cardinality over all the vertex-cuts of $G$ otherwise. It is well known that $\kappa(G)\leq\lambda(G)\leq\delta(G)$, where $\delta(G)$ is the minimum vertex-degree of $G$.

There are two properties of $p$-factor-critical graphs.

\begin{thm}[\cite{Favarvon,Yu}]\label{k-fc} A graph $G$ is $p$-factor-critical if and
only if $c_{0}(G-X)\leq|X|-p$ for all $X\subseteq V(G)$ with
$|X|\geq p$.
\end{thm}

\begin{thm}[\cite{Favarvon}]\label{pfc} If a graph $G$ is $p$-factor-critical with $1\leq p<|V(G)|$, then $\kappa(G)\geq p$ and $\lambda(G)\geq p+1$.
\end{thm}

Let $X$ be a subset of $V(G)$. Denoted by $\mathscr{C}_{G-X}$ the set of the components of $G-X$. $X$ is called to be {\em matchable} to $\mathscr{C}_{G-X}$ if the bipartite graph $G_{X}$, which arises from $G$ by contracting the components in $\mathscr{C}_{G-X}$ to single vertices and deleting all the edges in $E(X)$, contains a matching covering $X$. The following general result will be used.

\begin{thm}[\cite{Diestel}]\label{s} Every graph $G$ contains a set $X$ of vertices
with the following properties:\\
$(1)$ $X$ is matchable to $\mathscr{C}_{G-X}$;\\
$(2)$ Every component of $G-X$ is factor-critical.\\
Given any such set $X$, the graph $G$ contains a perfect matching if and only if $|X|=|\mathscr{C}_{G-X}|$.
\end{thm}

The {\em girth} of a graph $G$ with a cycle is the length of a shortest cycle in $G$ and the {\em odd girth} of a non-bipartite graph $G$ is the length of a shortest odd cycle in $G$. The girth and odd girth of $G$ are denoted by $g(G)$ and $g_{0}(G)$, respectively. $l$-cycle means a cycle of length $l$. We present two useful lemmas as follows.

\begin{Lemma}[\cite{Mantel}]\label{triangle} Let $G$ be a graph with $g_{0}(G)>3$. Then
$|E(G)|\leq\frac{1}{4}|V(G)|^2$.
\end{Lemma}

\begin{Lemma}[\cite{Andrasfai}]\label{order-girth} Let $G$ be a $k$-regular graph. If $g_{0}(G)\geq5$, then $|V(G)|\geq kg_{0}(G)/2$.
\end{Lemma}

Now we list some useful properties of vertex-transitive graphs as follows.

\begin{thm}[\cite{Mader}]\label{econnectivity} Let $G$ be a connected vertex-transitive $k$-regular graph. Then $\lambda(G)=k$.
\end{thm}

\begin{thm}[\cite{Watkins}]\label{vtc} Let $G$ be a connected vertex-transitive $k$-regular graph. Then $\kappa(G)>\frac{2}{3}k$.
\end{thm}

\begin{Lemma}[\cite{Watkins}]\label{vtc2} Let $G$ be a connected vertex-transitive $k$-regular graph. If $\kappa(G)<k$, then $\kappa(G)=m\tau(G)$ for
some integer $m\geq2$, where
\begin{center}
$\tau(G)$ $=$ min\{min\{$|V(P)|:$ $P$ is a component of $G-X$\}: $X$ is
a minimum vertex-cut of $G$\}.
\end{center}
\end{Lemma}

\begin{Lemma}[\cite{Watkins}]\label{connectivity} Let $G$ be a connected
vertex-transitive $k$-regular graph with $k=4$ or $6$. Then $\kappa(G)=k$.
\end{Lemma}

An {\em imprimitive block} of $G$ is a proper non-empty subset $X$ of $V(G)$ such that for any automorphism $\varphi$ of $G$, either $\varphi(X)=X$ or $\varphi(X)\cap X=\emptyset$.

\begin{Lemma}[\cite{R. Tindell}]\label{block} Let $G$ be a vertex-transitive graph and $X$ be an imprimitive block of $G$. Then $G[X]$ is also vertex-transitive.
\end{Lemma}

\begin{thm}[\cite{Heuvel}]\label{small} Let $G$ be a connected vertex-transitive $k$-regular graph of order $n$.
Let $S$ be a subset of $V(G)$ chosen such that
$\frac{1}{3}(k+1)\leq|S|\leq\frac{1}{2}n$, $d(S)$ is as small as
possible, and, subject to these conditions, $|S|$ is as small as
possible. If $d(S)<\frac{2}{9}(k+1)^2$, then $S$ is an imprimitive
block of $G$.
\end{thm}

\begin{cor}[\cite{Heuvel}]\label{small2} Let $G$ be a connected vertex-transitive $k$-regular graph of order $n$. Let $S$ be a subset of $V(G)$ chosen such that $1<|S|\leq\frac{1}{2}n$, $d_{G}(S)$ is as small as possible, and,
subject to these conditions, $|S|$ is as small as possible. If
$d_{G}(S)<2(k-1)$, then $d_{G}(S)=|S|\geq k$ and $d_{G[S]}(v)=k-1$
for all $v\in S$.
\end{cor}

\begin{cor}\label{2subset} Let $G$ be a connected vertex-transitive $k$-regular graph. Suppose $g(G)>3$ or $|V(G)|<2k$. Then $d_{G}(X)\geq2k-2$ for every $X\subseteq V(G)$ with $2\leq|X|\leq |V(G)|-2$.
\end{cor}
\begin{proof}
If $k=2$, then it is trivial. Now suppose $k\geq3$ and that there
is a subset $X\subseteq V(G)$ with $2\leq|X|\leq |V(G)|-2$ such that $d_{G}(X)<2k-2$. Let $S$ be a subset of $V(G)$ chosen such that
$1<|S|\leq\frac{1}{2}|V(G)|$, $d_{G}(S)$ is as small as possible, and,
subject to these conditions, $|S|$ is as small as possible. Then
$d_{G}(S)\leq d_{G}(X)<2k-2$. By Corollary \ref{small2}, $d_{G}(S)=|S|\geq k$
and $d_{G[S]}(v)=k-1$ for all $v\in S$. As $2k-3<\frac{2}{9}(k+1)^2$, $S$ is an imprimitive block of $G$ by Theorem \ref{small}. Then $|S|$ is a divisor of $|V(G)|$,
which implies $|V(G)|\geq2|S|\geq2k$. Thus $g(G)>3$. Noting that
$|E(S)|=\frac{1}{2}(k-1)|S|\leq\frac{1}{4}|S|^2$ by Lemma
\ref{triangle}, we have $d_{G}(S)=|S|\geq2k-2$, a contradiction.
\end{proof}

A subset $X$ of $V(G)$ is called an independent set of $G$ if any two vertices in $X$ are not adjacent. The maximum cardinality of independent sets of $G$ is the independent number of $G$, denoted by $\alpha(G)$.

\begin{Lemma}\label{independent} Let $G$ be a non-bipartite vertex-transitive $k$-regular graph. Then $\alpha(G)\leq\frac{1}{2}|V(G)|-\frac{k}{4}$ if $g_{0}(G)\geq5$, and $\alpha(G)\leq\frac{1}{3}|V(G)|$ if $g_{0}(G)=3$.
\end{Lemma}
\begin{proof} Let $X$ be a maximum independent set of $G$ and set $g_{0}:=g_{0}(G)$. Noting that $G$ is regular and non-bipartite, we have $|X|<|\overline{X}|$. Set $t=|\overline{X}|-|X|$. Since $G$ is vertex-transitive, the number of $g_{0}$-cycles of $G$ containing any given vertex in $G$ is constant. Let $q$ be this constant number and let $m$ be the number of all the $g_{0}$-cycles of $G$. Note that each $g_{0}$-cycle of $G$ contains at most $(g_{0}-1)/2$ vertices in $X$ and at least $(g_{0}+1)/2$ vertices in $\overline{X}$. We have $q|X|\leq\frac{1}{2}m(g_{0}-1)$ and $q|\overline{X}|\geq\frac{1}{2}m(g_{0}+1)$, which implies $qt=q(|\overline{X}|-|X|)\geq m$.

We know $q|V(G)|=mg_{0}$ by the vertex-transitivity of $G$. Then $qt\geq
m=\frac{q}{g_{0}}|V(G)|$, implying $t\geq\frac{|V(G)|}{g_{0}}$. If $g_{0}=3$, then $\alpha(G)=\frac{1}{2}(|V(G)|-t)\leq\frac{1}{3}|V(G)|$. If $g_{0}\geq5$, then $|V(G)|\geq kg_{0}/2$ by Lemma \ref{order-girth}, which implies $\alpha(G)=\frac{1}{2}(|V(G)|-t)\leq\frac{1}{2}|V(G)|-\frac{k}{4}$.
\end{proof}

A graph $G$ is called {\em trivial} if $|V(G)|=1$.

\begin{Lemma}\label{girthX} Let $G$ be a connected non-bipartite vertex-transitive graph. Let $X$ be an independent set of $G$. Suppose that $G-X$ has $|X|-t$ trivial components, where $t$ is a positive integer. Then $g_{0}(G)\geq\frac{2|X|}{t}+1$.
\end{Lemma}
\begin{proof} Let $Y$ be the set of vertices in the trivial components of $G-X$ and set $g_{0}:=g_{0}(G)$. Let $n_{i,j}$ be the number of $g_{0}$-cycles of $G$ which contain exactly $i$ vertices in $X$ and $j$ vertices in $Y$. Set $s=\frac{1}{2}(g_{0}-1)$. Since $X$ and $Y$ are independent sets of $G$, each $g_{0}$-cycle of $G$ contains at most $s$ vertices in $X$ and contains less vertices in $Y$ than in $X$. Let $q$ be the number of $g_{0}$-cycles of $G$ containing any given vertex in $G$. We have $\sum_{0\leq j<i\leq s}in_{i,j}=q|X|$ and $\sum_{0\leq j<i\leq s}jn_{i,j}=q|Y|=q(|X|-t)$. Then $q|X|=\sum_{0\leq j<i\leq s}in_{i,j}\leq\sum_{0\leq j<i\leq s}s(i-j)n_{i,j}=s(\sum_{0\leq j<i\leq s}in_{i,j}-\sum_{0\leq j<i\leq s}jn_{i,j})=sqt=\frac{1}{2}(g_{0}-1)qt$, which implies
$g_{0}\geq\frac{2|X|}{t}+1$.
\end{proof}

\begin{Lemma}\label{lessnumedge} Let $G$ be a vertex-transitive
graph with a triangle. Then the number of trivial components of $G-X$ is not larger than $|E(X)|$ for each subset $X\subseteq V(G)$.
\end{Lemma}
\begin{proof} Let $Y$ be the set of vertices in the trivial components of $G-X$. Suppose $|Y|>|E(X)|$. Let $q$ be the number of triangles of $G$ containing any given vertex in $G$. Note that there are $q|Y|$ triangles of $G$ containing vertices in $Y$. As $|Y|>|E(X)|$, it implies that $G[X]$ has an edge $e$ which is contained in more than $q$ triangles. This means that more than $q$ triangles containing both ends of $e$, a contradiction.
\end{proof}

\begin{Lemma}\label{K36} Let $G$ be a connected triangle-free vertex-transitive $6$-regular graph of even order. Suppose that there are $3$ distinct vertices with the same neighbors. Then $G$ is bipartite.
\end{Lemma}
\begin{proof} Suppose, to the contrary, that $G$ is non-bipartite. Then $g_{0}:=g_{0}(G)\geq5$. Let $C=u_{0}u_{1}\dots u_{g_{0}-1}u_{0}$ be a $g_{0}$-cycle of $G$. For any pair of vertices $u$ and $v$ in $V(C)$, $N(u)\neq N(v)$. So for each $u_{i}\in V(C)$ there are two distinct vertices $u'_{i}$ and $u''_{i}$ in $\overline{V(C)}$ such that $N(u_{i})=N(u'_{i})=N(u''_{i})$ by the vertex-transitivity of $G$. Set $U_{i}=\{u_{i},u'_{i},u''_{i}\}$. Then $U_{i}$ is an independent set of $G$ and $U_{i}\cap U_{j}=\emptyset$ for $j\neq i$. Noting that $u_{i}$ and $u_{i+1}$ are adjacent, every vertex in $U_{i}$ is adjacent to every vertex in $U_{i+1}$, where $i+1$ is an arithmetic on modular $g_{0}$. Since $G$ is 6-regular and connected,  $|V(G)|=|\bigcup_{i=0}^{g_{0}-1}U_{i}|=3g_{0}$, which implies that $|V(G)|$ is odd, a contradiction.
\end{proof}

\section{$\lambda_{s}$-atoms of vertex-transitive graphs}

In this section, we will present the concept of $\lambda_{s}$-atoms \cite{Holtkamp,Yang} of graphs in investigating the $s$-restricted edge-connectivity of graphs. The $s$-restricted edge-connectivity of graphs was proposed by F\`{a}brega and Fiol \cite{Fabrega}.

For a connected graph $G$ and some positive integer $s$, an edge-cut $F$ of $G$ is said to be an {\em $s$-restricted edge-cut} of $G$ if every component of $G-F$ has at least $s$ vertices. The minimum cardinality of $s$-restricted edge-cuts of $G$ is the {\em $s$-restricted edge-connectivity} of $G$, denoted by $\lambda_{s}(G)$. By the definition of $\lambda_{s}(G)$, we can see that $\lambda(G)=\lambda_{1}(G)\leq\lambda_{2}(G)\leq\lambda_{3}(G)\cdots$ as long as these parameters exists.

A proper subset $X$ of $V(G)$ is called a {\em $\lambda_{s}$-fragment} of $G$ if $\nabla(X)$ is an $s$-restricted edge-cut of $G$ with minimum cardinality. We can see that for every $\lambda_{s}$-fragment $X$ of $G$, $G[X]$ and $G[\overline{X}]$ are connected graphs of order at least $s$. A $\lambda_{s}$-fragment of $G$ with minimum cardinality is called a {\em $\lambda_{s}$-atom} of $G$.

\begin{Lemma}\label{5rec} Let $G$ be a connected triangle-free vertex-transitive graph of degree $k\geq5$. For an integer $4\leq s\leq8$, suppose $\lambda_{s}(G)\leq3k$. Let $S$ be a $\lambda_{s}$-atom of $G$.\\
(\textsl{a}) For $X\subseteq V(G)$ with $|X|\geq s$ and $|\overline{X}|\geq s$, we have $d_{G}(X)\geq\lambda_{s}(G)$. Furthermore, $d_{G}(X)>\lambda_{s}(G)$ if $G[X]$ or $G[\overline{X}]$ is disconnected.\\
(b) For $A\subseteq S$ with $1\leq|A|\leq|S|-s$, we have $d_{G[S]}(A)>\frac{1}{2}d_{G}(A)$.\\
(c) For each $\lambda_{s}$-atom $T$ of $G$ with $S\neq T$ and $S\cap
T\neq\emptyset$, we have $d_{G}(S\cap T)+d_{G}(S\cup
T)\leq2\lambda_{s}(G)$, $d_{G}(S\backslash T)+d_{G}(T\backslash
S)\leq2\lambda_{s}(G)$, $|S\cap T|\leq s-1$ and $|S\backslash T|\leq
s-1$.
\end{Lemma}
\begin{proof} (a) If $G[X]$ and $G[\overline{X}]$ are connected, then $\nabla(X)$ is an $s$-restricted edge-cut of $G$ and hence $d_{G}(X)\geq\lambda_{s}(G)$. Thus it only needs to show $d_{G}(X)>\lambda_{s}(G)$ if $G[X]$ or $G[\overline{X}]$ is disconnected.

Suppose that $G[X]$ is disconnected. If each component of $G[X]$ has less than 4 vertices, then $d_{G}(X)=k|X|-2|E(X)|\geq k|X|-2(|X|-2)\geq(k-2)s+4>3k\geq\lambda_{s}(G)$. Then we assume that $G[X]$ has a component $H_{1}$ with at least 4 vertices. If each component of $G[\overline{V(H_{1})}]$ has less than 4 vertices, then $d_{G}(X)>d_{G}(H_{1})=d_{G}(\overline{V(H_{1})})>\lambda_{s}(G)$. Then we assume further that $G[\overline{V(H_{1})}]$ has a component $H_{2}$ with at least 4 vertices. We know that $G[\overline{V(H_{2})}]$ is connected as $G$ is connected, which implies that $\nabla(H_{2})$ is 4-restricted edge-cut of $G$. Noting that $\lambda(G)=k$ by Theorem \ref{econnectivity}, we have $d_{G}(X)\geq\lambda(G)+d_{G}(H_{1})\geq k+d(V(H_{2}))\geq k+\lambda_{4}(G)$.

So $d(X)>\lambda_{4}(G)$. Next we consider the case that $5\leq s\leq8$. Set $\tau_{s}(G)=\textrm{min}\{d(A): A\subseteq V(G), 4\leq|A|\leq s-1\}$. Then $\lambda_{4}(G)\geq\textrm{min}\{\lambda_{s}(G),\tau_{s}(G)\}$. For each subset $A\subseteq V(G)$ with $4\leq|A|\leq7$, noting that $|E(A)|\leq\frac{1}{4}|A|^{2}$ by Lemma \ref{triangle}, we have $d(A)=k|A|-2|E(A)|\geq k|A|-\frac{1}{2}|A|^{2}>2k$. Hence $\tau_{s}(G)>2k$. If $\lambda_{s}(G)>2k$, then $d(X)\geq k+\lambda_{4}(G)>k+2k\geq\lambda_{s}(G)$. If $\lambda_{s}(G)\leq2k$, then, noting $\textrm{min}\{\lambda_{s}(G),\tau_{s}(G)\}\leq\lambda_{4}(G)\leq\lambda_{s}(G)$, we have $d(X)\geq k+\lambda_{4}(G)=k+\lambda_{s}(G)>\lambda_{s}(G)$.

(b) To the contrary, suppose $d_{G[S]}(A)\leq\frac{1}{2}d_{G}(A)$.
Then $d_{G}(S\backslash A)=d_{G}(S)-(d_{G}(A)-2d_{G[S]}(A))\leq
d_{G}(S)=\lambda_{s}(G)$. By (a), $G[S\backslash A]$ and
$G[\overline{S}\cup A]$ are connected. Hence $\nabla(S\backslash A)$
is an $s$-restricted edge-cut of $G$. By the minimality of
$\lambda_{s}$-atoms of $G$, we have $d_{G}(S\backslash
A)>\lambda_{s}(G)$, a contradiction.

(c) By the well-known submodular inequality (see \cite{Biggs} for
example), we have that $d_{G}(S\cap T)+d_{G}(S\cup T)\leq
d_{G}(S)+d_{G}(T)=2\lambda_{s}(G)$ and $d_{G}(S\backslash
T)+d_{G}(T\backslash
S)=d_{G}(S\cap\overline{T})+d_{G}(S\cup\overline{T})\leq
d_{G}(S)+d_{G}(\overline{T})=2\lambda_{s}(G)$. Next we show $|S\cap
T|\leq s-1$ and $|S\backslash T|\leq s-1$. Clearly, they hold if
$|S|=s$. So we may assume $|S|>s$.

Suppose $|S\cap T|\geq s$. Then $d_{G}(S\cap
T)=d_{G}(S)+2d_{G[S]}(S\backslash T)-d_{G}(S\backslash
T)>d_{G}(S)=\lambda_{s}(G)$ by (b). Noting $|\overline{S\cup
T}|\geq|V(G)|-|S|-|T|+|S\cap T|\geq s$, we have $d_{G}(S\cup
T)\geq\lambda_{s}(G)$ by (a). Hence $d_{G}(S\cap T)+d_{G}(S\cup
T)>2\lambda_{s}(G)$, a contradiction. Thus $|S\cap T|\leq s-1$.

If $|S\backslash T|=|T\backslash S|\geq s$, then we can similarly
obtain $d_{G}(S\backslash T)>\lambda_{s}(G)$ and $d_{G}(T\backslash
S)>\lambda_{s}(G)$ by (b), which implies $d_{G}(S\backslash
T)+d_{G}(T\backslash S)>2\lambda_{s}(G)$, a contradiction. Thus
$|S\backslash T|\leq s-1$.
\end{proof}

\begin{Lemma}\label{5rec14to15} Let $G$ be a connected triangle-free vertex-transitive $5$-regular graph of even order. For $s=5$ or $6$, suppose
$\lambda_{s}(G)=s+9$. Then $|S|\geq s+5$ for a $\lambda_{s}$-atom $S$ of $G$.
\end{Lemma}
\begin{proof} Suppose to the contrary that $|S|<s+5$. As $s+9=d_{G}(S)=5|S|-2|E(S)|$, $|S|$ and $s$ have different parities. Hence $|S|\geq s+1$. By Lemma \ref{5rec}(b), $\delta(G[S])\geq3$. If $|S|=s+1$, then $2|E(S)|\geq\delta(G[S])|S|\geq3|S|$, which implies $d_{G}(S)=5|S|-2|E(S)|\leq2|S|=2s+2<s+9$, a contradiction. Thus $|S|= s+3$. Let $R$ be the set of vertices $u$ in $S$ with $d_{G[S]}(u)=3$. By Lemma \ref{5rec}(b), $E(R)=\emptyset$. Noting $3s+9\leq\sum_{u\in S}d_{G[S]}(u)=2|E(S)|=5|S|-\lambda_{s}(G)=4s+6$, we have $|R|\geq|S|-(4s+6-3s-9)=6$. Since $s=5$ or 6, $d_{G[S]}(R)=3|R|\geq18>5(s-3)\geq d_{G[S]}(S\backslash R)$, a contradiction.
\end{proof}

\begin{Lemma}\label{no4fc} Let $G$ be a bicritical graph. If $G$ is not
$4$-factor-critical, then there is a subset $X\subseteq V(G)$ with
$|X|\geq4$ such that $c_{0}(G-X)=|X|-2$ and every component of $G-X$
is factor-critical.
\end{Lemma}
\begin{proof} Since $G$ is not 4-factor-critical, there is a set
$X_{1}$ of four vertices of $G$ such that $G-X_{1}$ has no
perfect matchings. By Theorem \ref{s}, $G-X_{1}$ has a vertex set
$X_{2}$ such that $X_{2}$ is matchable to $\mathscr{C}_{G-X_{1}-X_{2}}$ and
every component of $G-X_{1}-X_{2}$ is factor-critical. Set
$X=X_{1}\cup X_{2}$. Then $c_{0}(G-X)=|\mathscr{C}_{G-X}|>|X_{2}|=|X|-4$. Since $G$ is bicritical, we have
$c_{0}(G-X)\leq|X|-2$ by Theorem \ref{k-fc}. Hence
$|X|-4<c_{0}(G-X) \leq|X|-2$. Noting that $c_{0}(G-X)$ and $|X|$ have the
same parity, we have $c_{0}(G-X)=|X|-2$.
\end{proof}

In the rest of this section, we always suppose that $G$ is a
connected non-bipartite vertex-transitive graph of degree $k\geq5$
and even order, but $G$ is not 4-factor-critical. Also we always use
the following notation. Let $X$ be a subset of $V(G)$ with
$|X|\geq4$ such that $c_{0}(G-X)=|X|-2$ and every component of $G-X$
is factor-critical. By Theorem \ref{vtg} and Lemma \ref{no4fc}, such
subset $X$ exists. Let $H=H_{1}$, $H_{2}$, $\dots$, $H_{p}$ be the
nontrivial components of $G-X$. For a positive integer $m$, let
$[m]$ denote the set $\{1,2,\dots,m\}$.

\begin{Lemma}\label{component} We have $p\geq1$. Furthermore, if $g(G)>3$, then\\
(\textsl{a}) $p=1$ if $\lambda_{5}(G)>2k$,\\
(b) $|X|\geq7$ and $|V(H)|\geq9$ if $\lambda_{5}(G)>4k-8$ and $5\leq k\leq6$, and\\
(c) $|X|\geq10$ and $|V(H)|\geq15$ if $\lambda_{6}(G)\geq14$ and
$k=5$.
\end{Lemma}
\begin{proof} If $p=0$, then $|V(G)|=2|X|-2\geq2k-2\geq8$ and
$\alpha(G)\geq|\overline{X}|=\frac{1}{2}|V(G)|-1>\textrm{max}\{\frac{1}{3}|V(G)|,
\frac{1}{2}|V(G)|-\frac{k}{4}\}$, which contradicts Lemma \ref{independent}. Thus $p\geq1$.

Next we suppose $g(G)>3$. For each $i\in[p]$, we have $|V(H_{i})|\geq5$ as $H_{i}$ is triangle-free and factor-critical.

Suppose $\lambda_{5}(G)>2k$. By Lemma \ref{5rec}(a),
$d(H_{i})\geq\lambda_{5}(G)$ for each $i\in[p]$. We have
$2pk<p\lambda_{5}(G)\leq\sum_{i=1}^{p}d(H_{i})=d(X)-k(c_{0}(G-X)-p)\leq
k(p+2)$, which implies $p<2$. Thus $p=1$. It proofs (a).

Suppose $\lambda_{5}(G)>4k-8$ and $5\leq k\leq6$. We know $p=1$ by
(a). Assume $k=6$. Noting that $G$ is non-bipartite, it follows by
Lemma \ref{K36} that $|X|\geq7$. As $d(H)\leq3k$ and $H$ is
triangle-free and factor-critical, we have $|V(H)|\geq9$. Assume
next $k=5$. Note $|V(G)|=|V(H)|+2|X|-3\geq12$. By Lemma
\ref{5rec}(a), $d(A)\geq\lambda_{5}(G)>12$ for every subset
$A\subseteq V(G)$ with $|A|=6$, which implies that $G$ has no
subgraphs isomorphic to the complete bipartite graph $K_{3,3}$. By
the vertex-transitivity of $G$, it follows that $G$ has also no
subgraphs isomorphic to $K_{2,5}$. So $|X|\geq7$. If
$E(X)=\emptyset$, then $g_{0}(G)\geq7$ by Lemma \ref{girthX}, which
implies $|V(H)|\geq13$. If $E(X)\neq\emptyset$, then $d(H)=13$,
which implies $|V(H)|\geq9$. Hence the statement (b) holds.

Now we suppose $\lambda_{6}(G)\geq14$ and $k=5$. Then
$\lambda_{5}(G)\geq$min$\{\lambda_{6}(G),5k-12\}=13$. We know $p=1$
by (a). By the above argument, we know $|X|\geq7$, $|V(H)|\geq9$ and
that $G$ has no subgraphs isomorphic to $K_{2,5}$ or $K_{3,3}$. By
Lemma \ref{5rec}(a), $d(V(H)\cup A)\geq\lambda_{6}(G)$ and
$d(V(H)\backslash A)\geq\lambda_{6}(G)$ for every subset $A\subseteq
V(G)$ with $|A|\leq2$. It implies that $E(X)=\emptyset$,
$|\nabla(u)\cap\nabla(H)|\leq3$ for each $u\in V(G)$ and each of $X$
and $V(H)$ has at most one vertex $v$ with
$|\nabla(v)\cap\nabla(H)|=3$. Set $Y=\overline{V(H)\cup X}$.

Suppose $|X|=7$. Then $X$ has one vertex $u_{1}$ with 3 neighbors in $V(H)$ and other vertices in $X$ has exactly two neighbors in $V(H)$. Choose a vertex $u_{2}\in X\backslash\{u_{1}\}$ and a vertex $u_{3}\in Y\backslash N(u_{1})$. Since $G$ is vertex-transitive, there is an automorphism $\varphi_{1}$ of $G$ such that $\varphi_{1}(u_{3})=u_{2}$. Noting that $|N(v)\cap N(u_{3})|\geq3$ for each $v\in Y$, we have $\varphi_{1}(Y)\subseteq X$, which implies $|\nabla(v)\cap\nabla(H)|=3$ for each
$v\in N(u_{2})\cap V(H)$, a contradiction.

Suppose $8\leq|X|\leq9$. Then there are two vertices $u_{4}$ and $u_{5}$ in $X$ with
$|N(u_{4})\cap V(H)|=2$ and $|N(u_{5})\cap V(H)|\leq1$. Since $G$ is vertex-transitive, there is an automorphism $\varphi_{2}$ of $G$ such that $\varphi_{2}(u_{5})=u_{4}$. Then $\varphi_{2}(Y)\cap V(H)\neq\emptyset$ and $|\varphi_{2}(Y)\cap Y|\geq2$. As $G$ has no subgraphs isomorphic to $K_{2,5}$ or $K_{3,3}$, it follows that $|\varphi_{2}(X)\cap X|\geq6$. Hence $\varphi_{2}(Y)\subseteq V(H)\cup Y$ and $\varphi_{2}(X)\subseteq V(H)\cup X$. Noting that $|\nabla(u)\cap\nabla(H)|\leq3$ and $N(u)\subseteq\varphi_{2}(X)$ for each $u\in\varphi_{2}(Y)\cap V(H)$, we have $|\varphi_{2}(X)\cap V(H)|\geq2$. Notice that each of $X$ and $V(H)$ has at most one vertex $v$ with $|\nabla(v)\cap\nabla(H)|=3$. We know $d_{G[\varphi_{2}(X\cup Y)]}(\varphi_{2}(X)\cap V(H))\geq3$, which implies $|\varphi_{2}(Y)\cap V(H)|\geq3$. It follows that $N_{H}(\varphi_{2}(Y)\cap V(H))\geq3$. Now we have $|\varphi_{2}(X)\cap X|=6$ and $|\varphi_{2}(X)\cap V(H)|=3$ as $|\varphi_{2}(X)|=|X|\leq9$. It follows that $G[\varphi_{2}(X\cup Y)\cap V(H)]$ contains a subgraph isomorphic to $K_{3,3}$ if $|\varphi_{2}(Y)\cap V(H)|\geq4$ and $G[\varphi_{2}(X\cup Y)\backslash V(H)]$ contains a subgraph isomorphic to $K_{3,3}$ otherwise, a contradiction.

Thus $|X|\geq10$. Then $g_{0}(G)\geq9$ by Lemma \ref{girthX}. Let
$C$ be a $g_{0}(H)$-cycle of $H$. Then $g_{0}(H)\geq g_{0}(G)\geq9$
and $|N_{H}(v)\cap V(C)|\leq2$ for each $v\in V(H)\backslash V(C)$.
Noting $15=d(V(H))=5|V(H)|-2|E(H)|$, it is easy to verify
$|V(H)|\geq15$. It proofs (c).
\end{proof}

\begin{Lemma}\label{6to7rec} Suppose $k=5$, $\lambda_{6}(G)=\lambda_{7}(G)=12$ and $g(G)>3$. For a $\lambda_{7}$-atom $S$ of $G$, we have that $S$ is an imprimitive block of $G$.
\end{Lemma}
\begin{proof} Suppose, to the contrary, that $S$ is not an imprimitive block of $G$. Then there is an automorphism $\varphi_{1}$ of $G$ such that $\varphi_{1}(S)\neq S$ and $\varphi_{1}(S)\cap S\neq\emptyset$. Set $T=\varphi_{1}(S)$. By Lemma \ref{5rec}(c), we have $|S\cap T|\leq6$ and $|S\backslash T|\leq6$, which implies $|S|\leq12$. Noting that $12=d(S)=5|S|-2|E(S)|$, $|S|$ is an even integer. By Lemma \ref{5rec}(b), $\delta(G[S])\geq3$. For each $u\in\overline{S}$, we have $d_{G}(S\cup\{u\})\geq\lambda_{6}(G)$ by Lemma \ref{5rec}(a), which implies $|N_{G}(u)\cap S|\leq2$. As $\lambda_{6}(G)\geq\lambda_{5}(G)\geq\lambda_{4}(G)\geq\textrm{min}\{4k-8,5k-12,\lambda_{6}(G)\}=12$, we have $\lambda_{5}(G)=\lambda_{4}(G)=12$. By Lemma \ref{component}, $p=1$. By Lemma \ref{5rec}(a), we have $d_{G}(H)\geq\lambda_{5}(G)=12$. Then either $d_{G}(H)=13$ and $|E(X)|=1$, or $d_{G}(H)=15$ and $E(X)=\emptyset$.

\begin{figure}[h]
\begin{center}
\includegraphics[scale=0.7]{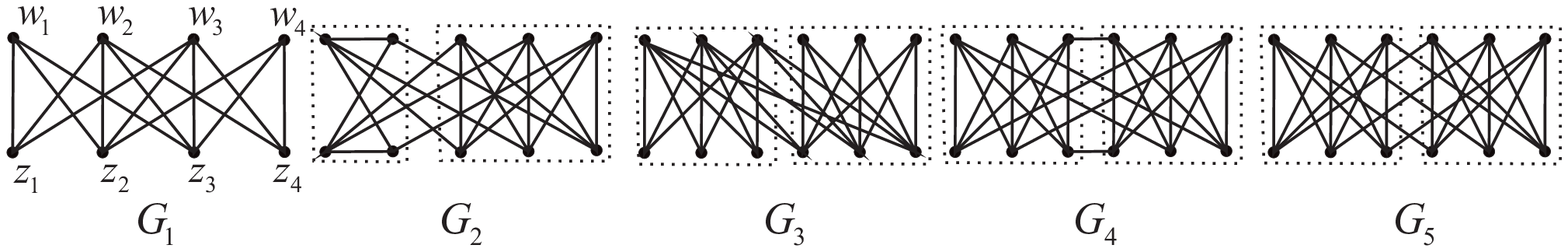}
\\{\small{Figure 1. Some possible cases of $G[S]$. In each $G_{i}$, $2\leq i\leq5$, the two graphs in the virtual boxes correspond to $G[S\cap T]$ and $G[S\backslash T]$.}}
\end{center}
\end{figure}

\textbf{Case 1.} $|S|=8$.

We have $|E(S)|=\frac{1}{2}(5|S|-\lambda_{6}(G))=14$. It is easy to verify that $G[S]$ is isomorphic to $G_{1}$ in Figure 1. Label $G[S]$ as in $G_{1}$ and set $W=\{w_{1},w_{2},w_{3},w_{4}\}$. As $|N_{G}(u)\cap S|\leq2$ for each $u\in\overline{S}$, $G$ has no vertex $v$ different from $w_{1}$ such that $N_{G}(v)=N_{G}(w_{1})$. Hence $G$ has no subgraphs isomorphic to $K_{2,5}$ by the vertex-transitivity of $G$.

\vskip 2mm
\textbf{Claim 1.} \emph{Each edge in $G$ is contained in a $4$-cycle of $G$. }
\vskip 2mm

Suppose that $G$ has an edge contained in no 4-cycles of $G$. Since
$G$ is vertex-transitive, each vertex in $G$ is incident with an
edge contained in no 4-cycles of $G$ and there is an automorphism
$\varphi_{2}$ of $G$ such that $\varphi_{2}(w_{1})=w_{2}$. As each
edge in $G[S]$ is contained in a 4-cycle, we have
$\varphi_{2}(N_{G[S]}(w_{1}))\subseteq N_{G[S]}(w_{2})$ and
$N_{G[S]}(\varphi_{2}(z_{i}))\subseteq\varphi_{1}(S)$ for each
$i\in\{2,3\}$. It implies $|S\cap\varphi_{2}(S)|\geq7$. On the other
hand, noting $\varphi_{2}(S)\neq S$, we have
$|S\cap\varphi_{2}(S)|\leq6$ by Lemma \ref{5rec}(c), a
contradiction. Thus Claim 1 holds.

\vskip 2mm
\textbf{Claim 2.} \emph{For any vertex $x\in V(G)$ with $2\leq|\nabla(x)\cap\nabla(H)|\leq3$ such that $d_{G[X]}(u)=0$ for each $u\in(\{x\}\cup N_{G}(x))\cap X$, there is a subset $A\subseteq N_{G}(x)$ with $|A|\geq|\nabla(x)\cap\nabla(H)|-1$ and a vertex $y\in V(G)\backslash\{x\}$ such that $\{xu,yu\}\subseteq\nabla(H)$ and $|\nabla(u)\cap\nabla(H)|\geq3$ for each $u\in A$.}
\vskip 2mm

Since $G$ is vertex-transitive, there is an automorphism $\varphi_{3}$ of $G$ such that $\varphi_{3}(w_{2})=x$. Let $T_{1}$ be one of $X$ and $V(H)$ such that $x\in T_{1}$, and let $T_{2}$ be the other of $X$ and $V(H)$. Then $\varphi_{3}(w_{3})\in T_{1}$ and $|\varphi_{3}(N_{G[S]}(w_{2}))\cap T_{2}|\geq|\nabla(x)\cap\nabla(H)|-1$. If $|\varphi_{3}(N_{G[S]}(w_{2}))\cap T_{2}|\leq2$ or $\varphi_{3}(W)\subseteq T_{1}$, then we choose $A$ to be $\varphi_{3}(N_{G[S]}(w_{2}))\cap T_{2}$. If $|\varphi_{3}(N_{G[S]}(w_{2}))\cap T_{2}|=3$ and $\varphi_{3}(W)\backslash T_{1}\neq\emptyset$, then $|\varphi_{3}(W)\cap T_{1}|=3$ and $\{\varphi_{3}(z_{2}),\varphi_{3}(z_{3})\}\subseteq T_{2}$. In the second case, we choose $A$ to be $\{\varphi_{3}(z_{2}),\varphi_{3}(z_{3})\}$. Then $A$ and $\varphi_{3}(w_{3})$ are a subset and a vertex which satisfy the condition. Thus Claim 2 holds.

\vskip 2mm

\textbf{Subcase 1.1.} Suppose first that $d_{G}(H)=13$.

Let $x_{1}x_{2}$ be the edge in $E(X)$. We know $|X|\geq6$ and
$|V(H)|\geq7$. By Lemma \ref{5rec}(a), $d_{G}(V(H)\cup
A)\geq\lambda_{4}(G)$ and $d_{G}(V(H)\backslash
A)\geq\lambda_{4}(G)$ for each subset $A\subseteq V(G)$ with
$|A|\leq2$, which implies that $|\nabla(u)\cap\nabla(H)|\leq3$ for
each $u\in V(G)$ and each of $X$ and $V(H)$ has at most one vertex
$v$ with $|\nabla(v)\cap\nabla(H)|=3$. Hence it follows by Claim 2
that $|\nabla(u)\cap\nabla(H)|\leq2$ for each $u\in
X\backslash\{x_{1},x_{2}\}$. By Claim 2 again, it follows that
$|\nabla(u)\cap\nabla(H)|\leq1$ for each $u\in V(H)\backslash
N_{G}(\{x_{1},x_{2}\})$.

We claim $|\nabla(u)\cap\nabla(H)|\leq2$ for each $u\in N_{G}(\{x_{1},x_{2}\})\cap V(H)$. Otherwise, suppose that there is a vertex $u_{0}\in N_{G}(\{x_{1},x_{2}\})\cap V(H)$ with $|\nabla(u_{0})\cap\nabla(H)|=3$. Since $G$ is vertex-transitive, there is an automorphism $\varphi_{4}$ of $G$ such that $\varphi_{4}(w_{2})=u_{0}$. It implies that there is a vertex $u_{1}\in\varphi_{4}(N_{G[S]}(w_{2})\cap(X\backslash\{x_{1},x_{2}\})$ such that $|\nabla(u_{1})\cap\nabla(H)|=3$, a contradiction.

Thus it follows by Claim 2 that $|\nabla(u)\cap\nabla(H)|\leq1$ for each $u\in X\backslash\{x_{1},x_{2}\}$. Noting $|N_{G}(\{x_{1},x_{2}\})\cap V(H)|\leq5$, we have $|\nabla(N_{G}(\{x_{1},x_{2}\})\cap V(H))\cap\nabla(H)|\leq10$ by the claim in the previous paragraph. Hence there is an edge $x_{3}x_{4}\in\nabla(H)$ such that $x_{3}\in X\backslash\{x_{1},x_{2}\}$ and $|\nabla(x_{3})\cap\nabla(H)|=|\nabla(x_{4})\cap\nabla(H)|=1$. Then $x_{3}x_{4}$ is contained in no 4-cycles of $G$, contradicting Claim 1. Hence Subcase 1.1 cannot occur.

\textbf{Subcase 1.2.} Now suppose $d_{G}(H)=15$.

Notice that $G$ has no subgraphs isomorphic to $K_{2,5}$. We know $|X|\geq6$. Next we show $|V(H)|\geq9$. Let $O_{i}$ be the set of vertices $u$ in $G$ with $|\nabla(u)\cap\nabla(H)|=i$ for $1\leq i\leq5$. If $|X|\geq7$, then $g_{0}(G)\geq7$ by Lemma \ref{girthX}, which implies $|V(H)|\geq13$. Then we assume $|X|=6$. As $G$ has no subgraphs isomorphic to $K_{2,5}$, we have $|O_{3}\cap X|=3$ and $|O_{2}\cap X|=3$. Noting $g(G)>3$, it follows that $|V(H)|\neq5$. By Claim 2, $|O_{3}\cap V(H)|\geq2$, which implies $|V(H)|\neq7$. Hence $|V(H)|\geq9$.

By Lemma \ref{5rec}(a), $d_{G}(V(H)\cup A)\geq\lambda_{4}(G)$ and
$d_{G}(V(H)\backslash A)\geq\lambda_{4}(G)$ for each subset
$A\subseteq V(G)$ with $|A|\leq4$. It implies $O_{5}=\emptyset$,
$|O_{4}\cap X|\leq1$, $|O_{3}\cap X|\leq3$, $|O_{3}\cap V(H)|\leq3$
and $|O_{4}\cap X|\cdot|O_{3}\cap X|=0$.

We claim $O_{4}=\emptyset$. Otherwise, suppose $O_{4}\neq\emptyset$. Noting that $\delta(H)\geq2$ as $H$ is factor-critical, we have $O_{4}\subseteq X$. Now we know $|O_{4}|=1$ and $O_{3}\cap X=\emptyset$. It follows by Claim 2 that $O_{3}\cap V(H)=\emptyset$ and $O_{2}\subseteq N_{G}(O_{4})$. Noting $|\nabla(N_{G}(O_{4})\cap V(H))|\leq8$, there is an edge $x_{5}x_{6}\in\nabla(H)$ with $\{x_{5},x_{6}\}\subseteq O_{1}$. Then $x_{5}x_{6}$ is contained in no 4-cycles of $G$, contradicting Claim 1.

Let $F_{1}$ be the subgraph of $G$ with vertex set $\bigcup_{i=1}^{3}O_{i}$ and edge set
$\nabla(H)$ and let $F_{2}$ be the subgraph of $F_{1}$ which is induced by $O_{3}$. By Claim 2,
$\delta(F_{2})\geq2$. Hence $F_{2}$ is connected. Then $F_{1}$ is connected by Claims 1 and 2. Let $t$ be the number of vertices $u$ in $F_{2}$ with $d_{F_{2}}(u)=2$. We have $15=|E(F_{1})|\leq|E(F_{2})|+2t=6|O_{3}|-3|E(F_{2})|$ by Claim 2. It follows that $|O_{3}|=6$ and $6\leq|E(F_{2})|\leq7$.

Assume $|E(F_{2})|=6$. Then $F_{2}$ is a 6-cycle. For each $u\in O_{3}\cap X$, there is a vertex $y_{u}\in X\backslash O_{3}$ such that $N_{F_{2}}(u)\subseteq N_{G}(y_{u})$ by Claim 2. It implies that there is a vertex $y\in X\backslash O_{3}$ such that $O_{3}\cap V(H)\subseteq N_{G}(y)$, which contradicts $|O_{3}\cap X|\leq3$.

Assume $|E(F_{2})|=7$. Noting $|E(F_{1})\backslash E(F_{2})|=8$, it follows by Claim 2 that there is a vertex $u_{1}\in V(F_{1})\backslash O_{3}$ with $d_{F_{1}}(u_{1})=2$ and we know $|N_{F_{1}}(u_{1})\cap O_{3}|=1$ and $d_{F_{1}}(N_{F_{1}}(u_{1})\backslash O_{3})=1$. Let $u_{2}$ be the vertex in $N_{F_{1}}(u_{1})\cap O_{3}$. It is easy to see that there is no vertex $u'$ in $G$ such that $|N_{G}(u'_{1})\cap N_{G}(u_{1})|=4$. Noting $|N_{G}(w_{2})\cap N_{G}(w_{3})|=4$, it implies that there is no automorphism $\varphi$ of $G$ such that $\varphi(w_{2})=u_{1}$, which contradicts the vertex-transitivity of $G$.

\textbf{Case 2.} \emph{$|S|=10$ or $12$.}

\vskip 2mm
\textbf{Claim 3.} \emph{For any given two distinct $\lambda_{7}$-atoms $S_{1}$ and $S_{2}$ of $G$ with $S_{1}\cap S_{2}\neq\emptyset$, $G[S_{1}\cap S_{2}]$ and $G[S_{1}\backslash S_{2}]$ are isomorphic to $K_{3,3}$ or $K_{2,2}$.}
\vskip 2mm

By Lemma \ref{5rec}(c), we have $d_{G}(S_{1}\cap
S_{2})+d_{G}(S_{1}\cup S_{2})\leq2\lambda_{7}(G)$,
$d_{G}(S_{1}\backslash S_{2})+d_{G}(S_{2}\backslash
S_{1})\leq2\lambda_{7}(G)$, $|S_{1}\cap S_{2}|\leq6$ and
$|S_{1}\backslash S_{2}|\leq6$. Then $|S_{1}\cap S_{2}|\geq4$ and
$|S_{1}\backslash S_{2}|\geq4$. By Lemma \ref{5rec}(a), each of
$d_{G}(S_{1}\cap S_{2})$, $d_{G}(S_{1}\cup S_{2})$,
$d_{G}(S_{1}\backslash S_{2})$ and $d_{G}(S_{2}\backslash S_{1})$ is
not less than $\lambda_{4}(G)$. Noting
$\lambda_{4}(G)=\lambda_{7}(G)=12$, we have $d_{G}(S_{1}\cap
S_{2})=d_{G}(S_{1}\backslash S_{2})=12$. Then $G[S_{1}\cap S_{2}]$
and $G[S_{1}\backslash S_{2}]$ are isomorphic to $K_{3,3}$ or
$K_{2,2}$. So Claim 3 holds.

\vskip 2mm

Let $R_{i}$ be the set of vertices $u$ in $S$ with $d_{G[S]}(u)=i$
for $3\leq i\leq5$. By Lemma \ref{5rec}(b), $E(G[R_{3}])=\emptyset$.

\vskip 2mm
\textbf{Claim 4.} \emph{$R_{5}=\emptyset$, or $G[R_{5}]$ is a $6$-cycle and $|S|=12$.}
\vskip 2mm

Suppose $R_{5}\neq\emptyset$. It only needs to show that $|S|=12$ and $G[R_{5}]$ is a $6$-cycle. Assume $R_{4}\neq\emptyset$. Choose a vertex $u\in R_{4}$ and a vertex $v\in R_{5}$. Let $\varphi_{5}$ be an automorphism of $G$ such that $\varphi_{5}(u)=v$. Then $\varphi_{5}(N_{G[S]}(u))\subseteq N_{G[S]}(v)$, which contradicts that $G[\varphi_{5}(S)\cap S]$ is isomorphic to $K_{3,3}$ or $K_{2,2}$ by Claim 3. Thus $R_{4}=\emptyset$. Noting $|R_{3}|+|R_{5}|=|S|$ and $3|R_{3}|+5|R_{5}|=2|E(S)|=5|S|-12$, we have $|R_{3}|=6$. For any two vertices $u',u''\in R_{5}$, it follows by Claim 3 that
$\varphi(S)=S$ for every
automorphism $\varphi$ of $G$ with $\varphi(u')=u''$. Hence $G[R_{5}]$ is $r$-regular, for some integer $r$. Then $18=3|R_{3}|=d_{G[S]}(R_{3})=d_{G[S]}(R_{5})=(5-r)(|S|-6)$, which implies $|S|=12$ and $r=2$. Hence $G[R_{5}]$ is a 6-cycle and Claim 4 is proved.

\vskip 2mm

By Claim 3, $G[S\cap T]$ and $G[S\backslash T]$ are isomorphic to $K_{3,3}$ or $K_{2,2}$. Noting $E(G[R_{3}])=\emptyset$, we have by Claim 4 that $G[S]$ is isomorphic to $G_{2}$, $G_{3}$, $G_{4}$ or $G_{5}$ in Figure 1.


\vskip 2mm
\textbf{Claim 5.} \emph{Each vertex in $G$ is contained in exactly two distinct $\lambda_{7}$-atoms of $G$.}
\vskip 2mm

By the vertex-transitivity of $G$, it only needs to show that $S'=S$ or $S'=T$ for a $\lambda_{7}$-atom $S'$ of $G$ with $S'\cap S\cap T\neq\emptyset$. Suppose $S'\neq S$ and $S'\neq T$. From Figure 1, we can see that $S$ has no subset $A$ different from $S\cap T$ and $S\backslash T$ such that $G[A]$ is isomorphic to $K_{3,3}$. Hence it follows by Claim 3 that $S'\cap S=S\cap T=S'\cap T$. Then $12=d_{G}(S\cap T)\geq d_{G[S]}(S\cap T)+d_{G[T]}(S\cap T)+d_{G[S']}(S\cap T)=18$, a contradiction. So Claim 5 holds.

\vskip 2mm

Suppose $|S|=10$. Then $G[S]$ is isomorphic to $G_{2}$. By Claims 3 and 5, there is a $\lambda_{7}$-atom $S''$ of $G$ such that $S''\cap S=S\backslash T$. Choose a vertex $u_{2}\in S\backslash T$ and a vertex $u_{3}\in S\cap T$. Noting that $G[S\backslash T]$ is not isomorphic to $G[S\cap T]$, we know by Claim 5 that there is no automorphism $\varphi$ of $G$ such that $\varphi(u_{2})=u_{3}$, a contradiction.

Suppose next $|S|=12$. Then $G[S]$ is isomorphic to $G_{3}$, $G_{4}$ or $G_{5}$. Let $V_{1}$, $V_{2}$, $\dots$, $V_{m}$ be all subsets of $V(G)$ which induce subgraphs of $G$ isomorphic to $K_{3,3}$. Noting that $G[S\cap T]$ and $G[S\backslash T]$ are isomorphic to $K_{3,3}$, it follows by Claims 3 and 5 that $V_{1}$, $V_{2}$, $\dots$, $V_{m}$ form a partition of $V(G)$ and for each $V_{i}$ there are exactly two elements $j_{1},j_{2}\in\{1,2,\dots,m\}\backslash\{i\}$ such that $G[V_{i}\cup V_{j_{1}}]$ and $G[V_{i}\cup V_{j_{2}}]$ are isomorphic to $G[S]$. We denote $V_{i}\sim V_{j}$ if $G[V_{i}\cup V_{j}]$ is isomorphic to $G[S]$, and assume $V_{1}\sim V_{2}\sim\dots\sim V_{m}\sim V_{1}$. If $G[S]$ is isomorphic to $G_{3}$, then it is easy to verify that $G$ is bipartite, a contradiction.
Thus $G[S]$ is isomorphic to $G_{4}$ or $G_{5}$.

Assume that there is some $V_{q}\subseteq\overline{V(H)}$. If $G[S]$ is isomorphic to $G_{4}$, then $N_{G}(V_{q}\backslash X)\cap V_{q-1}\subseteq X$, which implies
$|E(X)|\geq|E(V_{q-1})\cap E(X)|\geq2$, a contradiction. Thus $G[S]$ is isomorphic to $G_{5}$. Let $V_{j}$ be chosen such that $V_{j}\cap V(H)\neq\emptyset$ and $|j-q|$ is as small as possible. Then $|V_{j}\cap X|=3$ and $|N(u)\cap X|\geq4$ for each
$u\in V_{j}\cap V(H)$, which contradicts that $\delta(H)\geq2$.

Then we assume that $V_{i}\cap V(H)\neq\emptyset$ for $1\leq i\leq m$. Then $|V_{i}\cap X|>|V_{i}\backslash(V(H)\cup X)|$ if $V_{i}\cap X\neq\emptyset$. Choose some $V_{q'}$ which contains vertices in $V(G)\backslash(V(H)\cup X)$. Then $V_{q'-1}\cap X\neq\emptyset$ and $V_{q'+1}\cap X\neq\emptyset$. Noting $c_{0}(G-X)=|X|-2$, it follows that for each $i\in[m]$, $|V_{i}\cap X|=|V_{i}\backslash(V(H)\cup X)|+1$ if $i\in\{q'-1,q',q'+1\}$ and $|V_{i}\cap X|=\emptyset$ otherwise. Then $|V_{q'}\backslash(V(H)\cup X)|=2$. Hence $|V_{q'-1}\cap X)|=|V_{q'+1}\cap X)|=3$. Now we have $V_{q'-1}\sim V_{q'}\sim V_{q'+1}\sim V_{q'-1}$, which implies $V(G)=V_{q'-1}\cup V_{q'}\cup V_{q'+1}$ and $|V(H)|=3$. It follows that $g(G)=3$, a contradiction.
\end{proof}

\begin{Lemma}\label{5rec13} Suppose $k=5$, $\lambda_{5}(G)=\lambda_{6}(G)=13$ and $g(G)>3$. For a $\lambda_{6}$-atom $S$ of $G$, we have $|S|\geq 11$.
\end{Lemma}
\begin{proof} To the contrary, suppose $|S|<11$. Noting that $13=d(S)=5|S|-2|E(S)|$, $|S|$ is odd. Then $|S|\geq7$. By Lemma \ref{5rec}(b), $\delta(G[S])\geq3$. By Lemma \ref{component}, we have $p=1$, $|X|\geq7$ and $|V(H)|\geq9$. Hence $|V(G)|\geq20$.

Assume $|S|=7$. Then $|E(S)|=\frac{1}{2}(5|S|-13)=11$. If $G[S]$ is bipartite, then
$|E(S)|\geq\frac{1}{2}(|S|+1)\delta(G[S])\geq12$, a contradiction. Thus $G[S]$ is non-bipartite. Let $C$ be a shortest cycle of odd length in $G[S]$. Then $5\leq|V(C)|\leq7$. Noting that $|N_{G[S]}(u)\cap V(C)|\leq2$ for each $u\in S\backslash V(C)$, we have $|E(S)|\leq10$, a contradiction.

So $|S|=9$. Let $R_{i}$ be the set of vertices $u$ in $S$ with $d_{G[S]}(u)=i$ for $3\leq i\leq5$.

\vskip 2mm

\textbf{Claim 1.} \emph{For any automorphism of $\varphi$ of $G$ with $\varphi(R_{4}\cup R_{5})\cap(R_{4}\cup R_{5})\neq\emptyset$, either $\varphi(S)=S$ or $G[S\cap\varphi(S)]$ is isomorphic to $K_{2,3}$.}

\vskip 2mm

Suppose $\varphi(S)\neq S$. By Lemma \ref{5rec}(c),
$|S\cap\varphi(S)|\leq5$, $|S\backslash\varphi(S)|\leq5$ and
$d(S\cap\varphi(S))+d(S\cup\varphi(S))\leq2\lambda_{6}(G)$. Then
$4\leq|S\cap\varphi(S)|\leq5$ and
$|S\cup\varphi(S)|=|S|+|\varphi(S)|-|S\cap\varphi(S)|\leq14$. As
$|V(G)|\geq20$, we have $d(S\cup\varphi(S))\geq\lambda_{6}(G)$ by
Lemma \ref{5rec}(a). Hence
$d(S\cap\varphi(S))\leq\lambda_{6}(G)=13$. Noting
$|N_{G[\varphi(S)]}(u)\cap N_{G[S]}(u)|\geq3$ for each $u\in
\varphi(R_{4}\cup R_{5})\cap(R_{4}\cup R_{5})$, it follows that
$G[S\cap\varphi(S)]$ is isomorphic to $K_{2,3}$. So Claim 1 holds.

\vskip 2mm

By Claim 1, it follows that $G$ has no automorphism $\varphi$ such
that $\varphi(R_{4})\cap R_{5}\neq\emptyset$, which implies
$R_{4}=\emptyset$ or $R_{5}=\emptyset$. Noting
$\sum_{i=3}^{5}i|R_{i}|=2|E(S)|=32$ and
$\sum_{i=3}^{5}|R_{i}|=|S|=9$, we have $|R_{3}|=4$, $|R_{4}|=5$ and
$R_{5}=\emptyset$. By Lemma \ref{5rec}(b), $E(R_{3})=\emptyset$.
Hence $|E(R_{4})|=4$. As $g(G[S])\geq g(G)>3$, it is easy to verify
that $G[R_{4}]$ has a 4-cycle or is isomorphic to $K_{1,4}$. Let
$u_{1}$ and $u_{2}$ be two vertices in $R_{4}$ with
$d_{G[R_{4}]}(u_{1})<d_{G[R_{4}]}(u_{2})$. Since $G$ is
vertex-transitive, there is an automorphism $\psi$ of $G$ such that
$\psi(u_{2})=u_{1}$. By Claim 1, $G[\psi(S)\cap S]$ is isomorphic to
$K_{2,3}$. As $u_{1},u_{2}\in R_{4}$, we know $d_{G[\psi(S)\cap
S]}(u_{1})=3$. Note that $|N_{G[S]}(u)\cap N_{G[S]}(u_{1})|\leq2$
for each $u\in S\backslash\{u_{1}\}$ if $G[R_{4}]$ has a 4-cycle. It
follows that $G[R_{4}]$ is isomorphic to $K_{1,4}$. Since
$d_{G[\psi(S)\cap S]}(v)=2$ for each $v\in N_{G[\psi(S)\cap
S]}(u_{1})$, it follows that $N_{G[\psi(S)\cap S]}(u_{1})\subseteq
R_{3}$. It implies that the vertex in $R_{3}\backslash
N_{G[S]}(u_{1})$ has only two neighbors in $S$, which contradicts
$\delta(G[S])\geq3$.
\end{proof}

\begin{Lemma}\label{7rec} Suppose $k=5$, $\lambda_{6}(G)=\lambda_{7}(G)=14$ and $g(G)>3$. For a $\lambda_{7}$-atom $S$ of $G$, we have $|S|\geq 14$.
\end{Lemma}
\begin{proof} By Lemma \ref{component}, we have $p=1$, $|X|\geq10$ and $|V(H)|\geq15$. Hence $|V(G)|\geq32$. For $1\leq i\leq5$, let $O_{i}$ be the set of vertices $u$ in $G$ with $|\nabla(u)\cap(V(H))|=i$, and set $m_{i}=|O_{i}\cap X|$ and $n_{i}=|O_{i}\cap V(H)|$.
By Lemma \ref{5rec}(a), $d_{G}(V(H)\cup A)\geq\lambda_{6}(G)$ and
$d_{G}(V(H)\backslash A)\geq\lambda_{6}(G)$ for each subset $A$ of
$V(G)$ with $|A|\leq2$. Also noting that $d_{G}(H)$ is odd, it
follows that $d_{G}(H)=15$, $O_{4}\cup O_{5}=\emptyset$ and
$m_{3}\cdot n_{3}\leq1$. Hence $E(X)=\emptyset$. Then
$g_{0}(G)\geq9$ by Lemma \ref{girthX}.

Suppose $|S|<14$. As $5|S|-2|E(G[S])|=14$, $|S|$ is an even integer
with $8\leq|S|\leq12$. By Lemma \ref{5rec}(b), $\delta(G[S])\geq3$.
As $g_{0}(G)\geq9$, it follows that $G[S]$ is bipartite. By Lemma
\ref{5rec}(a), $d_{G}(S\cup\{u\})\geq\lambda_{6}(G)$ for each
$u\in\overline{S}$ and $d_{G}(A)\geq\lambda_{6}(G)$ for each subset
$A\subseteq V(G)$ with $|A|=6$. Hence $|N_{G}(u)\cap S|\leq2$ for
each $u\in\overline{S}$ and $G$ has no subgraphs isomorphic to
$K_{3,3}$.

\vskip 2mm
\textbf{Claim 1.} \emph{For any two distinct $\lambda_{7}$-atoms $S_{1}$ and $S_{2}$ of $G$ with $S_{1}\cap S_{2}\neq\emptyset$, we have $d_{G}(S_{1}\cap S_{2})\leq14$ and furthermore, $G[S_{1}\cap S_{2}]$ and $G[S_{1}\backslash S_{2}]$ are isomorphic to $K_{2,4}$ or $K_{3,3}-e$ if $|S|=12$, where $K_{3,3}-e$ is a subgraph of $K_{3,3}$ obtained by deleting an edge $e$ from $K_{3,3}$.}
\vskip 2mm

By Lemma \ref{5rec}(c), we have $|S_{1}\cap S_{2}|\leq6$,
$|S_{1}\backslash S_{2}|\leq6$, $d_{G}(S_{1}\cap
S_{2})+d_{G}(S_{1}\cup S_{2})\leq2\lambda_{7}(G)$ and
$d_{G}(S_{1}\backslash S_{2})+d_{G}(S_{2}\backslash
S_{1})\leq2\lambda_{7}(G)$. Noting $|V(G)|\geq32$, we have
$d_{G}(S_{1}\cup S_{2})\geq\lambda_{7}(G)$ by Lemma \ref{5rec}(a).
Hence $d_{G}(S_{1}\cap S_{2})\leq\lambda_{7}(G)=14$. Next assume
$|S|=12$. Then $|S_{1}\cap S_{2}|=|S_{1}\backslash S_{2}|=6$. By
Lemma \ref{5rec}(a), each of $d_{G}(S_{1}\cap S_{2})$,
$d_{G}(S_{1}\backslash S_{2})$ and $d_{G}(S_{2}\backslash S_{1})$ is
not less than $\lambda_{6}(G)$. Hence $d_{G}(S_{1}\cap
S_{2})=d_{G}(S_{1}\backslash S_{2})=14$. It implies that
$G[S_{1}\cap S_{2}]$ and $G[S_{1}\backslash S_{2}]$ are isomorphic
to $K_{2,4}$ or $K_{3,3}-e$. So Claim 1 holds.

\begin{figure}[h]
\begin{center}
\includegraphics[scale=0.6]{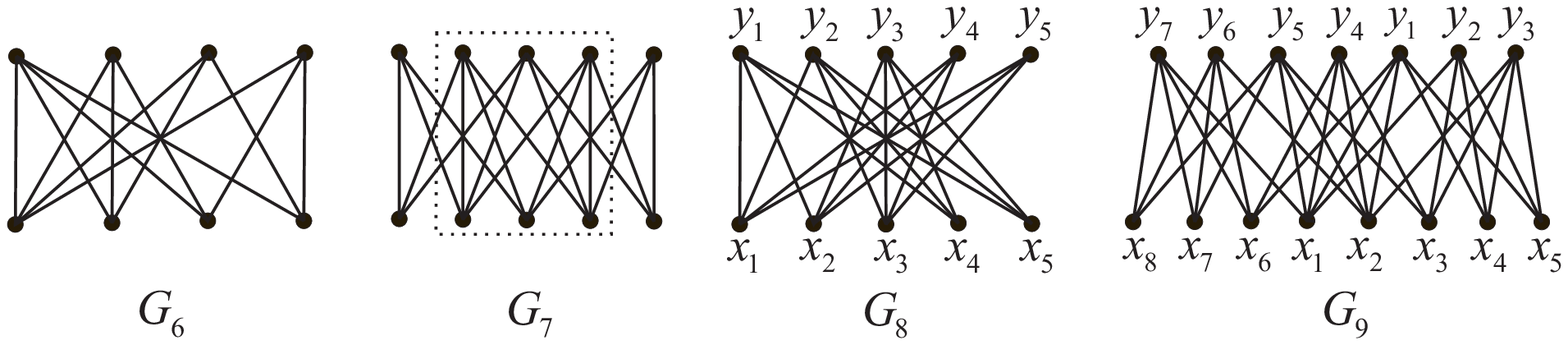}
\\{\small{Figure 2. The illustration in the proof of Lemma \ref{7rec}.}}
\end{center}
\end{figure}

\textbf{Case 1.} $|S|=8$.

As $G[S]$ is a bipartite graph with $|E(S)|=13$ and $\delta(G[S])\geq3$, $G[S]$ is isomorphic to $G_{6}$ in Figure 2. Let $v_{1}$, $v_{2}$ be the two vertices in $S$ with $d_{G[S]}(v_{1})=d_{G[S]}(v_{2})=4$ and choose a vertex $v_{3}\in N_{G[S]}(v_{1})\backslash\{v_{2}\}$.

We claim that each edge in $G$ is contained in a 4-cycle of $G$.
Otherwise, suppose that $G$ has an edge contained in no 4-cycles of
$G$. Since $G$ is vertex-transitive, each vertex in $G$ is incident
with an edge contained in no 4-cycles of $G$ and there is an
automorphism $\varphi_{1}$ of $G$ such that
$\varphi_{1}(v_{3})=v_{2}$. Clearly, $\varphi_{1}(S)\neq S$. Noting
that each edge in $G[S]$ is contained in a 4-cycle of $G[S]$, we
have $d_{G[\varphi_{1}(S)\cup S]}(u)\leq4$ for each
$u\in\varphi_{1}(S)\cup S$. Then
$\varphi_{1}(N_{G[S]}(v_{3})\subseteq N_{G[S]}(v_{2})$ and
$N_{G[S]}(\varphi_{1}(v_{1}))\subseteq\varphi_{1}(N_{G[S]}(v_{1}))$.
Noting that $|\varphi_{1}(S)\cap S|\leq6$ by Lemma \ref{5rec}(c) and
$d_{G}(\varphi_{1}(S)\cap S)\leq14$ by Claim 1,
$G[S\cap\varphi_{1}(S)]$ is isomorphic to $K_{3,3}-e$. As
$d_{G}(S\cup\varphi_{1}(S))\geq\lambda_{6}(G)=14$ by Lemma
\ref{5rec}(a), it follows that $G[S\cup\varphi_{1}(S)]$ is
isomorphic to $G_{7}$ in Figure 2, where the graph in the virtual
box corresponds to $G[S\cap\varphi_{1}(S)]$. Choose a vertex
$v_{4}\in S\cap\varphi_{1}(S)$ with
$d_{G[S\cap\varphi_{1}(S)]}(v_{4})=2$. Let $\varphi_{2}$ be an
automorphism of $G$ such that $\varphi_{2}(v_{1})=v_{4}$. Then
$\varphi_{1}(N_{G[S]}(v_{1}))=N_{G[S\cup\varphi_{1}(S)]}(v_{4})$ and
$\varphi_{1}(N_{G[S]}(v_{2}))\backslash(S\cup\varphi_{1}(S))\neq\emptyset$.
Then
$d_{G}(S\cup\varphi_{2}(S)\cup\varphi_{1}(N_{G[S]}(v_{2})))<14=\lambda_{6}(G)$,
contradicting Lemma \ref{5rec}(a). So our this claim holds.

For each $uv\in\nabla(H)$, noting that $uv$ is contained in a 4-cycle of $G$ by the previous claim, we have $|\nabla(u)\cap\nabla(H)|+|\nabla(v)\cap\nabla(H)|\geq3$. Hence $m_{1}\leq n_{2}+2n_{3}$. For each $u\in O_{2}\cup O_{3}$, there is an automorphism $\varphi_{3}$ of $G$ such that $\varphi_{3}(v_{1})=u$, which implies that there is a vertex $v\in\varphi_{3}(N_{G[S]}(v_{1}))$ such that $uv\in\nabla(H)$ and $|\nabla(v)\cap\nabla(H)|=3$. Hence $m_{2}\leq3n_{3}$ and $n_{2}\leq3m_{3}$. Noting $m_{3}\cdot n_{3}\leq1$, we have $15=\sum_{i=1}^{3}im_{i}\leq n_{2}+2n_{3}+6n_{3}+3m_{3}\leq6m_{3}+8n_{3}\leq14$, a contradiction.

\textbf{Case 2.} \emph{$|S|=10$ or $12$.}

Let $R_{i}$ be the set of vertices $u$ in $S$ with $d_{G[S]}(u)=i$
for $3\leq i\leq5$. Then $E(R_{3})=\emptyset$ by Lemma
\ref{5rec}(b). Let $Z$ and $W$ be the bipartition of $G[S]$ such
that $|Z|\leq|W|$. Noting
$\frac{1}{2}(5|S|-14)=|E(S)|\geq\delta(G[S])|W|\geq3|W|$, we have
$|W|<\frac{1}{2}|S|+2$.

\vskip 2mm
\textbf{Claim 2.} \emph{If $R_{5}\neq\emptyset$, then, for each $v\in R_{4}$, there is exactly one vertex $w$ in $S\backslash\{v\}$ with $N_{G[S]}(v)\subseteq N_{G[S]}(w)$.}
\vskip 2mm

Suppose $R_{5}\neq\emptyset$. Choose a vertex $u\in R_{5}$ and a
vertex $v\in R_{4}$. Let $\varphi_{4}$ be an automorphism of $G$
such that $\varphi_{4}(u)=v$. Then
$N_{G[S]}(v)\subseteq\varphi_{4}(N_{G[S]}(u))$. Noting that
$|S\cap\varphi_{4}(S)|\leq6$ by Lemma \ref{5rec}(c) and
$d_{G}(S\cap\varphi_{4}(S))\leq14$ by Claim 1,
$G[S\cap\varphi_{4}(S)]$ is isomorphic to $K_{2,4}$. It implies that
$S$ has a vertex $w$ different from $v$ with $N_{G[S]}(v)\subseteq
N_{G[S]}(w)$. As $G$ has no subgraphs isomorphic to $K_{3,3}$, such
vertex $w$ is unique. So Claim 2 holds.

\vskip 2mm
\textbf{Claim 3.} \emph{$|W|=|Z|$ and $R_{5}=\emptyset$.}
\vskip 2mm

Suppose, to the contrary, that $|W|>|Z|$, or $|W|=|Z|$ and $R_{5}\neq\emptyset$. As $E(R_{3})=\emptyset$, it follows that $|W|=6$ if $|S|=10$.

Assume $|W|=|Z|+2=7$. Noting $|E(S)|=23$, there is a vertex $v_{5}\in(R_{4}\cup R_{5})\cap W$ and a vertex $v_{6}\in R_{5}\cap Z$. Let $\varphi_{5}$ be an automorphism of $G$ such that $\varphi_{5}(v_{5})=v_{6}$. Then $\varphi_{5}(S)\neq S$ and $\varphi_{5}(N_{G[S]}(v_{5}))\subseteq N_{G[S]}(v_{6})$. Hence $G[S\cap\varphi_{5}(S)]$ is isomorphic to $K_{2,4}$ by Claim 1. It implies $|\varphi_{5}(W)\backslash S|=5$, contradicting that $G[\varphi_{5}(S)\backslash S]$ is isomorphic to $K_{2,4}$ or $K_{3,3}-e$ by Claim 1.

Assume $|W|=6$. If $|S|=10$, we know $|R_{4}\cap Z|=|R_{5}\cap Z|=2$ as $E(R_{3})=\emptyset$ and $|E(S)|=18$. If $|S|=12$, we know either $|R_{5}\cap Z|=2=|R_{4}\cap Z|+1$ or $|R_{5}\cap Z|=1=|R_{4}\cap Z|-2$ as $|E(S)|=23$. It follows by Claim 2 that there is a vertex $v_{7}\in R_{4}\cap Z$ and a vertex $v_{8}\in(R_{4}\cup R_{5})\backslash\{v_{7}\}$ such that $N_{G[S]}(v_{7})\subseteq N_{G[S]}(v_{8})$ and $(R_{5}\cap Z)\backslash\{v_{8}\}\neq\emptyset$. It implies that $G[S]$ has a subgraph isomorphic to $K_{3,3}$, a contradiction. So Claim 3 holds.

\vskip 2mm

\textbf{Subcase 2.1.} Suppose first that $|S|=10$.

By Claim 3, $G[S]$ is isomorphic to $G_{8}$ in Figure 2 and we label $G[S]$ as in $G_{8}$. Assume $x_{1}\in Z$ and $y_{1}\in W$.

\vskip 2mm
\textbf{Claim 4.} \emph{$|N_{G}(u)\cap N_{G}(v)|\leq3$ for any two distinct vertices $u$ and $v$ in $G$.}
\vskip 2mm

Suppose that there are two distinct vertices $u$ and $v$ in $G$ with $|N_{G}(u)\cap N_{G}(v)|\geq4$. Notice that $|N_{G}(u)\cap S|\leq2$ for each $u\in\overline{S}$. By the vertex-transitivity of $G$, for each $y_{i}\in\{y_{1},y_{2},y_{3}\}$ there is a vertex $y_{j}\in\{y_{1},y_{2},y_{3}\}\backslash\{y_{i}\}$ such that $|N_{G}(y_{i})\cap N_{G}(y_{j})|\geq4$. It follows that there is a vertex $w\in\overline{S}$ such that $\{y_{1},y_{2},y_{3}\}\subseteq N_{G}(w)$, a contradiction. So Claim 4 holds.

\vskip 2mm

Let $\varphi_{6}$ be an automorphism of $G$ such that
$\varphi_{6}(y_{5})=y_{1}$. Then $\varphi_{6}(S)\neq S$ and
$|\varphi_{6}(N_{G[S]}(y_{5}))$ $\cap$ $N_{G[S]}(y_{1})|\geq2$. Then
$|\varphi_{6}(S)\cap S|\leq6$ by Lemma \ref{5rec}(c) and
$d_{G}(\varphi_{6}(S)\cap S)\leq14$ by Claim 1. It follows that
$|\varphi_{6}(S)\cap W|\leq3$ and $|\varphi_{6}(S)\cap Z|\leq3$
since $G[S]$ has no subgraphs isomorphic to $K_{2,4}$ by Claim 4.

Assume $\varphi_{6}(N_{G[S]}(y_{5}))\cap\{x_{1},x_{2}\}\neq\emptyset$ and $\varphi_{6}(N_{G[S]}(y_{5}))\cap\{x_{4},x_{5}\}\neq\emptyset$. Then $|N_{G[\varphi_{6}(S)]}(u)\cap N_{G[S]}(u)|=3$ for each $u\in\varphi_{6}(N_{G[S]}(y_{5}))\cap\{x_{1},x_{2}\}$ and $|N_{G[\varphi_{6}(S)]}(v)\cap N_{G[S]}(v)|\geq2$ for each $v\in\varphi_{6}(N_{G[S]}(y_{5}))\cap\{x_{4},x_{5}\}$. It follows that $|\varphi_{6}(S)\cap W|=3$ and $|\varphi_{6}(S)\cap\{y_{4},y_{5}\}|=1$. Noting $2\leq|\varphi_{6}(S)\cap Z|\leq3$, we can see $d_{G}(\varphi_{6}(S)\cap S)>14$, a contradiction.

Assume $\varphi_{6}(N_{G[S]}(y_{5}))\cap N_{G[S]}(y_{1})=\{x_{4},x_{5}\}$. Then $\varphi_{6}(y_{4})\in\{y_{2},y_{3}\}\cup\overline{S}$, which implies that $|N_{G}(y_{1})\cap N_{G}(\varphi_{6}(y_{4}))|\geq4$ or $|N_{G}(x_{4})\cap N_{G}(x_{5})|\geq4$. It contradicts Claim 4.

Thus $\varphi_{6}(N_{G[S]}(y_{5}))\cap N_{G[S]}(y_{1})=\{x_{1},x_{2}\}$. By Claim 4, $\varphi_{6}(y_{4})\in\{y_{4},y_{5}\}$ and $\varphi_{6}(\{y_{1},y_{2},$ $y_{3}\})\cap W=\{y_{4},y_{5}\}\backslash\varphi_{6}(y_{4})$. Then $\{\varphi_{6}(x_{4}),\varphi_{6}(x_{5})\}\subseteq\overline{S}$. Assume $\varphi_{6}(y_{4})=y_{4}$. Set $\{y_{6},y_{7}\}=\varphi_{6}(\{y_{1},y_{2},$ $y_{3}\})\backslash W$, $\{x_{6}\}=\varphi_{6}(N_{G[S]}(y_{5}))\backslash N_{G[S]}(y_{1})$ and $\{x_{7},x_{8}\}=\{\varphi_{6}(x_{4}),\varphi_{6}(x_{5})\}$. Then the graph $G_{9}$ showed in Figure 2 is a subgraph of $G$.

We can see that each edge incident with $x_{1}$ is contained in a 4-cycle of $G$. Then, by the vertex-transitivity of $G$, each edge $uv\in\nabla(H)$ is contained in a 4-cycle of $G$, which implies $|\nabla(u)\cap\nabla(H)|\geq2$ or $|\nabla(v)\cap\nabla(H)|\geq2$. Hence there is a vertex $u'\in G$ with $2\leq|\nabla(u_{2})\cap\nabla(H)|\leq3$. Let $\varphi_{7}$ be an automorphism of $G$ such that $\varphi_{7}(y_{4})=u'$. It is easy to verify that either $\varphi_{7}(N_{G[\varphi_{6}(S)\cup S]}(y_{4}))$ has a vertex $u$ with $|\nabla(u)\cap\nabla(H)|\geq4$ or it have two vertices $v'$ and $v''$ with $\{u_{2}v',u_{2}v''\}\subseteq\nabla(H)$ and $|\nabla(v')\cap\nabla(H)|=|\nabla(v'')\cap\nabla(H)|=3$, contradicting that $O_{4}\cup O_{5}=\emptyset$ and $m_{3}\cdot n_{3}\leq1$.

\textbf{Subcase 2.2.} Now suppose $|S|=12$.

Noting $|E(G[S])|=23$, $G[S]$ is not regular. Let $\varphi_{8}$ be an automorphism of $G$ such that $\varphi_{8}(S)\neq S$ and $\varphi_{8}(S)\cap S\neq\emptyset$. Set $T=\varphi_{8}(S)$. It follows by Claims 1 and 3 that $d_{G[S\cup T]}(u)=5$ for each $u\in S\cap T$, each of $G[S\backslash T]$, $G[S\cap T]$ and $G[T\backslash S]$ is isomorphic to $K_{3,3}-e$ and $d_{G[S]}(v)=d_{G[T]}(v)=4$ for each $v\in S\cap T$ with $d_{G[S\cap T]}(v)=3$.

Let $v_{9}$ and $v_{10}$ be two vertices in $W\cap T$ with $d_{G[S\cap T]}(v_{9})=3=d_{G[S\cap T]}(v_{10})+1$. We know either
$d_{G[S]}(v_{10})=4$ or $d_{G[T]}(v_{10})=4$ and assume $d_{G[S]}(v_{10})=4$ without loss of generality. Let $\varphi_{9}$ be an automorphism of $G$ such that $\varphi_{9}(v_{9})=v_{10}$. Let $Q$ be one of $\varphi_{9}(S)$ and $\varphi_{9}(T)$ such that $Q\neq S$. Since $d_{G[Q]}(v_{10})=4$, we know $Q\neq T$. By Claims 1 and 3, each of $G[Q\cap S]$, $G[Q\backslash S]$, $G[Q\cap T]$ and $G[Q\backslash T]$ is isomorphic to $K_{3,3}-e$. Noting $d_{G[S]}(v_{10})=d_{G[Q]}(v_{10})=4$, we have $|N_{G[Q]}(v_{10})\cap S|=3$, which implies $2\leq|Q\cap S\cap T|\leq5$.

Assume $2\leq|Q\cap S\cap T|\leq3$. Noting that $G[Q\cap T]$ is isomorphic to $K_{3,3}-e$, we have $d_{G[Q\cap T]}(Q\cap S\cap T)>|Q\cap S\cap T|\geq d_{G[T]}(Q\cap S\cap T)$, a contradiction.

Assume $4\leq|Q\cap S\cap T|\leq5$. Then $|N_{G[Q]}(v_{10})\cap S\cap T|=2$. If $E(Q\cap S\cap\overline{T})=\emptyset$, then $d_{G[Q\cap S]}(Q\cap S\cap\overline{T})+d_{G[Q\backslash T]}(Q\cap S\cap\overline{T})\geq4|Q\cap S\cap\overline{T}|>|[Q\cap S\cap\overline{T},\overline{S}\cup T]|$, a contradiction. Thus $|Q\cap S\cap\overline{T}|=2$ and $|E(Q\cap S\cap\overline{T})|=1$. Then $d_{G[Q\cap S]}(Q\cap S\cap\overline{T})+d_{G[Q\backslash T]}(Q\cap S\cap\overline{T})\geq3+3>5\geq|[Q\cap S\cap\overline{T},\overline{S}\cup T]|$, a contradiction.
\end{proof}

\begin{Lemma}\label{8rec} Suppose $k=5$, $\lambda_{6}(G)\geq14$, $\lambda_{8}(G)=15$ and $g(G)>3$. For a $\lambda_{8}$-atom $S$ of $G$, we have $|S|\geq 15$.
\end{Lemma}
\begin{proof} By Lemma \ref{component}, we have $p=1$, $|X|\geq10$ and $|V(H)|\geq15$. By Lemma \ref{5rec}(a), $d_{G}(A)\geq\lambda_{6}(G)\geq14$, $d_{G}(V(H)\cup B)\geq\lambda_{8}(G)$ and $d_{G}(V(H)\backslash B)\geq\lambda_{8}(G)$ for any two subsets $A$ and $B$ of $V(G)$ with $|A|=6$ and $|B|\leq1$. It implies that $G$ has no subgraphs isomorphic to $K_{3,3}$, $d_{G}(H)=15$ and $|\nabla(u)\cap\nabla(H)|\leq2$ for each $u\in V(G)$. Hence $E(X)=\emptyset$ and there is an edge $u_{1}u_{2}\in\nabla(H)$ such that $N_{G}(u_{1})\cap X=\{u_{2}\}$. By Lemma \ref{girthX}, $g_{0}(G)\geq9$.

Suppose $|S|<15$. As $g_{0}(G)\geq9$ and
$15=\lambda_{8}(G)=d_{G}(S)=5|S|-2|E(G[S])|$, it follows that $|S|$
is odd and $G[S]$ is bipartite. By Lemma \ref{5rec}(b),
$\delta(G[S])\geq3$. Let $W$ and $Z$ be the bipartition of $G[S]$
such that $|W|>|Z|$. We have $|W|=\frac{1}{2}(|S|+1)$ if
$|S|\leq11$, and $7\leq|W|\leq8$ if $|S|=13$.

\textbf{Case 1.} \emph{There is a vertex $v_{1}$ in $S$ with $d_{G[S]}(v_{1})=5$.}

Let $R$ be one of $W$ and $Z$ such that $v_{1}\in R$. As
$\delta(G[S])\geq3$ and $|E(S)|=\frac{1}{2}(5|S|-2|E(G[S])|)$, it
follows that $N_{G[S]}(N_{G[S]}(v_{1}))=R$. Since $G$ is
vertex-transitive, there is an automorphism $\varphi_{1}$ of $G$
such that $\varphi_{1}(v_{1})=u_{2}$. Then $\varphi_{1}(R)\subseteq
X\cup V(H)$. Noting that $|\nabla(u)\cap\nabla(H)|\leq2$ for each
$u\in V(G)$, we have $\varphi_{1}(S\backslash R)\cap X=\emptyset$.
Notice that $G$ has no subgraphs isomorphic to $K_{3,3}$. We have
$|\varphi_{1}(R)\cap X|\geq4$ as $|N_{G}(u_{2})\backslash
V(H)|\geq3$ and $\delta(G[S])\geq3$. Then $|\varphi_{1}(S)\cap
V(H)|\leq6$ as $|S|\leq13$. It follows that
$d_{G[\varphi_{1}(S)]}(u_{1})=3$. Then
$d_{G[\varphi_{1}(S)]}(v)\geq4$ for each $v\in
N_{G[\varphi_{1}(S)]}(u_{1})$ by Lemma \ref{5rec}(b). Now we know
$|S|=13$, $|\varphi_{1}(R)\cap X|=4=|\varphi_{1}(R)\cap V(H)|+2$ and
$|\varphi_{1}(S\backslash R)\cap V(H)|=4=|\varphi_{1}(S\backslash
R)\backslash V(H)|+1$. Then $R=Z$ and $|N_{G}(u_{2})\cap V(H)|=2$.

Noting that $|\nabla(u)\cap\nabla(H)|\leq2$ for each $u\in V(G)$, we have $d_{G[\varphi_{1}(S)]}(u)\leq4$ for each $u\in\varphi_{1}(W)$. Since $\delta(G[S])\geq3$ and $G$ has no subgraphs isomorphic to $K_{3,3}$, two vertices in $\varphi_{1}(W)\backslash V(H)$ has exactly 3 neighbors in $\varphi_{1}(Z)\cap X$. So $d_{G[\varphi_{1}(S)]}(u)=4$ for each $u\in\varphi_{1}(W)\backslash N_{G}(u_{2})$ as $|E(S)|=25$. Then there is a vertex $u_{3}\in \varphi_{1}(Z)\cap X$ such that $\varphi_{1}(W)\backslash N_{G}(u_{2})\subseteq N_{G}(u_{3})$.

Assume $\varphi_{1}(Z)\cap V(H)=\{u_{4},u_{5}\}$. Let $\varphi_{2}$ be an automorphism of $G$ such that $\varphi_{2}(u_{4})=u_{2}$. Then $u_{1}\notin\varphi_{2}(N_{G[\varphi_{1}(S)]}(u_{4}))$ and $\varphi_{2}(\{u_{2},u_{3},u_{5}\})\subseteq X$, which implies $|\nabla(u)\cap\nabla(H)|\geq3$ for the vertex $u\in(N_{G}(u_{2})\cap V(H))\backslash\{u_{1}\}$, a contradiction.

\textbf{Case 2.} \emph{$d_{G[S]}(u)\leq4$ for each $u\in S$.}

If $|S|=13$, then, noting $|E(G[S])|=25$ and $5\leq|Z|\leq6$, there is a vertex $u\in Z$ with $d_{G[S]}(u)=5$, a contradiction. Thus $|S|\leq11$. There is a vertex $w\in W$ with $d_{G[S]}(w)=|W|-2$ such that $d_{G[S]}(u)=4$ for each $u\in N_{G[S]}(w)$. Choose a vertex $z\in N_{G[S]}(w)$.

We claim that the edge $u_{1}u_{2}$ is contained in a 4-cycle of
$G$. Suppose not. Since $G$ is vertex-transitive, each vertex in $G$
is incident with an edge contained in no 4-cycles of $G$ and there
is an automorphism $\varphi_{3}$ of $G$ such that
$\varphi_{3}(w)=z$. We know $\varphi_{3}(S)\neq S$. Noting that
$|N_{G[S]}(u)\cap N_{G[S]}(v)|\geq2$ for every subset
$\{u,v\}\subseteq Z$, each edge in $G[S]$ is contained in a 4-cycle
of $G[S]$. Hence $\varphi_{3}(N_{G[S]}(w))\subseteq N_{G[S]}(z)$ and
$N_{G[S]}(u)\subseteq\varphi_{3}(S)$ for each $u\in
\varphi_{3}(N_{G[S]}(w))$. By Lemma \ref{5rec}(c),
$|S\cap\varphi_{3}(S)|\leq7$ and
$d_{G}(S\cap\varphi_{3}(S))+d_{G}(S\cup\varphi_{3}(S))\leq2\lambda_{8}(G)$.
If $|S|=11$, then
$|S\cap\varphi_{3}(S)|\geq|\varphi_{3}(N_{G[S]}(w))\cup
N_{G[S]}(w)|=8$, a contradiction. Thus $|S|=9$. As $G$ has no
subgraphs isomorphic to $K_{3,3}$, we have
$Z=\bigcup_{u\in\varphi_{3}(N_{G[S]}(w))}N_{G[S]}(u)\subseteq\varphi_{3}(S)$.
Hence $|S\cap\varphi_{3}(S)|=7$ and $d_{G}(S\cap\varphi_{3}(S))=17$.
Noting that $d_{G}(S\cup\varphi_{3}(S))\geq\lambda_{8}(G)$ by Lemma
\ref{5rec}(a), we have
$d_{G}(S\cap\varphi_{3}(S))+d_{G}(S\cup\varphi_{3}(S))>2\lambda_{8}(G)$,
a contradiction.

Thus $|N_{G}(u_{2})\cap V(H)|=2$. Let $\varphi_{4}$ be an automorphism of $G$ such that $\varphi_{4}(z)=u_{2}$ if $|S|=9$, and $\varphi_{4}(w)=u_{2}$ if $|S|=11$. If $u_{1}\in\varphi_{4}(S)$, then $|Z|\geq d_{G[\varphi_{4}(S)]}(u_{1})-1+|N_{G[\varphi_{4}(S)]}(N_{G[\varphi_{4}(S)]}(u_{2})\backslash V(H))|\geq2+3=5$ if $|S|=9$, and $|W|\geq7$ if $|S|=11$, a contradiction. Thus $u_{1}\notin\varphi_{4}(S)$. Then $\varphi_{5}(Z)\subseteq X$ if $|S|=9$ and $\varphi_{5}(W)\subseteq X$ if $|S|=11$, which implies $|\nabla(u)\cap\nabla(H)|\geq3$ for the vertex $u\in(N_{G}(u_{2})\cap V(H))\backslash\{u_{1}\}$, a contradiction.
\end{proof}

\begin{Lemma}\label{5rec=16} Suppose $k=6$, $\lambda_{5}(G)=16$ and $g(G)>3$. For a $\lambda_{5}$-atom $S$ of $G$, we have $|S|\geq9$.
\end{Lemma}
\begin{proof} To the contrary, suppose $|S|\leq8$. As $\frac{1}{2}(6|S|-\lambda_{5}(G))=|E(S)|\leq\frac{1}{4}|S|^{2}$ by Lemma \ref{triangle}, we have $|S|\geq8$. Hence $|S|=8$ and $G[S]$ is isomorphic to $K_{4,4}$.

By Lemma \ref{component}, $p=1$. Then $|X|\geq7$ by Lemma \ref{K36}.
Noting that $d(H)\leq18$ and $H$ is triangle-free and
factor-critical, we have $|V(H)|\geq11$. Let $O_{i}$ be the set of
vertices $u$ in $G$ with $|\nabla(u)\cap\nabla(H)|=i$ for $4\leq
i\leq6$. By Lemma \ref{5rec}(a), we have $d(V(H)\cup
A)\geq\lambda_{5}(G)$ and $d(V(H)\backslash A)\geq\lambda_{5}(G)$
for each subset $A\subseteq V(G)$ with $|A|\leq3$, which implies
$d(H)\geq16$, $O_{5}\cup O_{6}=\emptyset$, $|O_{4}\cap X|\leq1$ and
$|O_{4}\cap V(H)|\leq1$.

Suppose that $S$ is an imprimitive block of $G$. Then the orbits
$S=S_{1}$, $S_{2}$, $\dots$, $S_{m}$ of $S$ under the automorphism
group of $G$ form a partition of $V(G)$. If $E(S_{i})\cap
E(X)\neq\emptyset$ for some $S_{i}$, then $d(H)=16$ and $|S_{i}\cap
V(H)|=6$, which implies $d(V(H)\cup S_{i})\leq14<\lambda_{5}(G)$, a
contradiction. Thus $E(S_{j})\cap E(X)=\emptyset$ for each $S_{j}$.
Noting $c_{0}(G-X)=|X|-2$, it follows that $|O_{4}|\geq3$, which
contradicts the fact that $|O_{4}|=|O_{4}\cap X|+|O_{4}\cap
V(H)|\leq2$.

Suppose next that $S$ is not an imprimitive block of $G$. Then there
is an automorphism $\varphi_{1}$ of $G$ such that
$\varphi_{1}(S)\neq S$ and $\varphi_{1}(S)\cap S\neq\emptyset$. Set
$T=\varphi_{1}(S)$. As $G$ is 6-regular, we have $\delta(G[S\cap
T])\geq2$. By Lemma \ref{5rec}(c), $|S\cap T|\leq4$. Hence $G[S\cap
T]$ is a 4-cycle of $G$. Assume $S\cap
T=\{v_{1},v_{2},v_{3},v_{4}\}$, where $N(v_{1})=N(v_{2})$ and
$N(v_{3})=N(v_{4})$.

By the vertex-transitivity of $G$, for each $u\in V(G)$ there is a vertex $u'$ different from $u$ such that $N(u')=N(u)$. Assume $E(X)\neq\emptyset$. Then $|E(X)|=1$ and let $u_{1}u_{2}$ be the edge in $E(X)$. We know that there is a vertex $u'_{1}$ in $V(H)$ with $N(u_{1}')=N(u_{1})$, which implies $|N(u_{1})\cap V(H)|=5$. Then $O_{5}\neq\emptyset$, a contradiction. Thus $E(X)=\emptyset$. As for each $u\in V(G)$ there is a vertex $u'$ different from $u$ such that $N(u')=N(u)$, it follows that there is a vertex $u_{3}\in X$ with $2\leq|N(u_{3})\cap V(H)|\leq4$. Let $\varphi_{2}$ be an automorphism of $G$ such that $\varphi_{2}(v_{1})=u_{3}$. If $\varphi_{2}(\{v_{3},v_{4}\})\backslash V(H)\neq\emptyset$, then $N(u_{3})\cap V(H)=\varphi_{2}(N(v_{1}))\cap V(H)\subseteq\bigcup_{i=4}^{6}O_{i}$. If $\varphi_{2}(\{v_{3},v_{4}\})\subseteq V(H)$, then $\varphi_{2}(\{v_{3},v_{4}\})\subseteq\bigcup_{i=4}^{6}O_{i}$. So $|(\bigcup_{i=4}^{6}O_{i})\cap V(H)|\geq2$, a contradiction.
\end{proof}

\begin{Lemma}\label{5rec18} Suppose $k=6$, $\lambda_{5}(G)=\lambda_{8}(G)=18$ and $g(G)>3$. For a $\lambda_{8}$-atom $S$ of $G$, we have $|S|\geq15$.
\end{Lemma}
\begin{proof}
To the contrary, suppose $8\leq|S|\leq14$. By Lemma \ref{component},
we have $p=1$, $|X|\geq7$ and $|V(H)|\geq9$. By Lemma \ref{5rec}(a),
we have $d_{G}(V(H)\cup A)\geq\lambda_{5}(G)$ and
$d_{G}(V(H)\backslash A)\geq\lambda_{5}(G)$ for each subset
$A\subseteq V(G)$ with $|A|\leq1$, which implies $d_{G}(H)=18$ and
$|\nabla(u)\cap\nabla(H)|\leq3$ for each $u\in V(G)$. Then
$g_{0}(G)\geq7$ by Lemma \ref{girthX}. It follows that $G[A]$ is
bipartite for each subset $A\subseteq V(G)$ with $|A|\leq13$ and
$d_{G}(A)=18$. Hence $|V(H)|\geq15$, and $G[S]$ is bipartite if
$|S|\leq13$. Then $|V(G)|\geq26$.

\textbf{Case 1.} $|S|=8$.

By Lemma \ref{5rec}(a), $d_{G}(A)\geq\lambda_{5}(G)$ for every
subset $A\subseteq V(G)$ with $7\leq|A|\leq8$, which implies
$\delta(G[S])\geq3$ and $G$ has no subgraphs isomorphic to
$K_{4,4}$. Noting that $|E(G[S])|=\frac{1}{2}(6|S|-18)=15$ and
$G[S]$ is bipartite, there is a vertex $u_{0}\in S$ with
$d_{G[S]}(u_{0})=3$ and $G[S\backslash\{u_{0}\}]$ is isomorphic to
$K_{3,4}$.

\begin{figure}[h]
\begin{center}
\includegraphics[scale=0.7]{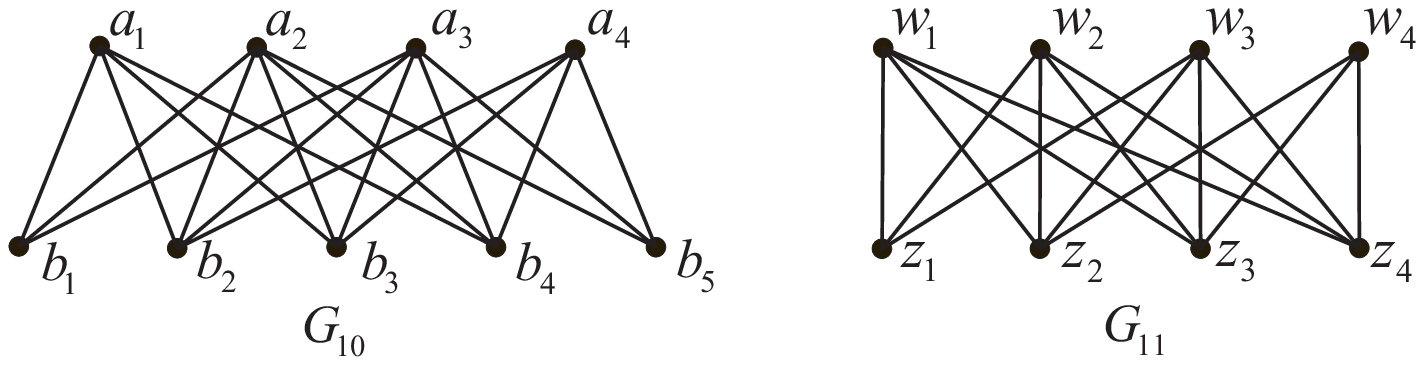}
\\{\small{Figure 3. The illustration in the proof of Lemma \ref{5rec18}.}}
\end{center}
\end{figure}

\textbf{Claim 1.} \emph{There are no two distinct vertices $u$ and $v$ in $G$ with $N_{G}(u)=N_{G}(v)$.}
\vskip 2mm

Suppose that $u_{1}$ and $u_{2}$ are two distinct vertices in $G$ with $N_{G}(u_{1})=N_{G}(u_{2})$. Let $x$, $y$ and $z$ be the 3 vertices in $S$ which have 4 neighbors in $S\backslash\{u_{0}\}$. Noting that $G$ has no subgraphs isomorphic to $K_{4,4}$, it follows by the vertex-transitivity of $G$ that for each vertex $u\in\{x,y,z\}$ there is a vertex $u'\in\{x,y,z\}\backslash\{u\}$ such that $N_{G}(u)=N_{G}(u')$. It follows that $N_{G}(x)=N_{G}(y)=N_{G}(z)$. Then $G$ is bipartite by Lemma \ref{K36}, a contradiction. So Claim 1 holds.

\vskip 2mm
\textbf{Claim 2.} \emph{$G$ has no subgraphs isomorphic to $K_{3,5}$.}
\vskip 2mm

Suppose that $u_{3}$, $u_{4}$ and $u_{5}$ are 3 distinct vertices in $G$ with $|N_{G}(u_{3})\cap N_{G}(u_{4})\cap N_{G}(u_{5})|=5$. By Claim 1 and the vertex-transitivity of $G$, it follows that for each $u\in N_{G}(u_{3})\cap N_{G}(u_{4})$ there are two distinct vertices $u',u''\in(N_{G}(u_{3})\cap N_{G}(u_{4}))\backslash\{u\}$ such that $|N_{G}(u)\cap N_{G}(u')\cap N_{G}(u'')|=5$. It implies that there is a vertex $v\in V(G)\backslash(\{u_{3},u_{4},u_{5}\})$ such that $|N_{G}(v)\cap N_{G}(u_{3})\cap N_{G}(u_{4})|\geq4$. So $G$ has a subgraph isomorphic to $K_{4,4}$, a contradiction. Claim 2 is proved.

\vskip 2mm \textbf{Claim 3.} \emph{$G$ has no subgraphs isomorphic
to $G_{10}$ in Figure $3$.} \vskip 2mm Suppose that $G_{10}$ is a
subgraph of $G$. Let $\varphi_{1}$ be an automorphism of $G$ such
that $\varphi_{1}(a_{2})=a_{1}$. Noting that $d_{G}(V(G_{10})\cup
A)\geq\lambda_{5}(G)$ for each subset $A\subseteq V(G)$ with
$|A|\leq1$ by Lemma \ref{5rec}(a), we have $G_{10}=G[V(G_{10})]$ and
$|N_{G}(u)\cap V(G_{10})|\leq3$ for each $u\in\overline{V(G_{10})}$.
We know $\varphi_{1}(a_{3})\in\{a_{2},a_{3}\}$ if
$|\varphi_{1}(N_{G_{10}}(a_{2}))\cap N_{G_{10}}(a_{1})|=4$. Hence
either each edge in $\nabla(a_{1})$ or each edge in
$\nabla(\varphi_{1}(a_{3}))$ is contained in a 4-cycle of $G$. By
the vertex-transitivity of $G$, each edge in $G$ is contained in a
4-cycle of $G$. It follows that
$|\nabla(u)\cap\nabla(H)|+|\nabla(v)\cap\nabla(H)|\geq3$ for each
edge $uv\in\nabla(H)$.

We claim that $|\nabla(u)\cap\nabla(H)|\leq2$ for each $u\in V(G)$. Otherwise, noting that $|\nabla(u)\cap\nabla(V(H))|\leq3$ for each $u\in V(G)$, we suppose that there is a vertex $u_{6}$ in $G$ with $|\nabla(u_{6})\cap\nabla(H)|=3$. Let $\varphi_{2}$ be an automorphism of $G$ such that $\varphi_{2}(b_{2})=u_{6}$. By considering what will $\varphi_{2}(V(G_{10}))$ be, we can obtain that there is a vertex $u\in\varphi_{2}(N_{G_{10}}(b_{2}))$ with $|\nabla(u)\cap\nabla(H)|\geq4$, a contradiction.

Thus there a vertex $u_{7}\in V(G)$ with $|\nabla(u_{7})\cap\nabla(H)|=2$. Let $\varphi_{3}$ be an automorphism of $G$ such that $\varphi_{3}(a_{2})=u_{7}$. Then there is a vertex $u\in\varphi_{3}(N_{G_{10}}(a_{2}))$ with $|\nabla(u)\cap\nabla(H)|\geq3$, a contradiction. So Claim 3 holds.

\vskip 2mm

By Claim 2, it follows that $G[S]$ is isomorphic to $G_{11}$ in Figure 3 and we label $G[S]$ as in $G_{11}$. Then $|N_{G}(u)\cap S|\leq2$ for each $u\in\overline{S}$ by Claims 2 and 3. Let $\varphi_{4}$ be an automorphism of $G$ such that $\varphi_{4}(z_{1})=z_{4}$. If $\varphi_{4}(N_{G[S]}(z_{1}))\subseteq N_{G[S]}(z_{4})$, then there is a vertex $u\in\varphi_{4}(S)\backslash S$ with $|N_{G}(u)\cap S|\geq3$, a contradiction. Thus $\varphi_{4}(N_{G[S]}(z_{1}))\backslash S\neq\emptyset$.

Assume $\varphi_{3}(N_{G[S]}(z_{1}))\cap N_{G[S]}(z_{1})=\{w_{i},w_{j}\}$. As $|N_{G}(u)\cap S|\leq2$ for each $u\in\varphi_{4}(N_{G[S]}(z_{1}))$ $\backslash S$, it follows that $|\varphi_{4}(\{z_{2},z_{3},z_{4}\})\backslash S|=2$. Then $N_{G}(w_{i})=N_{G}(w_{j})$, contradicting Claim 1.

Assume $\varphi_{3}(N_{G[S]}(z_{1}))\cap N_{G[S]}(z_{1})=\{w_{i'}\}$. Then $|\varphi_{4}(\{z_{2},z_{3},z_{4}\})\backslash S|=2$, which implies that each edge in $\nabla(w_{i'})$ is contained in a 4-cycle of $G$. Then each edge in $G$ is contained in a 4-cycle of $G$ by the vertex-transitivity of $G$. Thus there is a vertex $u_{8}\in V(G)$ with $2\leq|\nabla(u_{8})\cap\nabla(H)|\leq3$. Let $\varphi_{5}$ be an automorphism of $G$ such that $\varphi_{5}(z_{4})=u_{8}$. Noting $|N_{G}(w_{1})\cap N_{G}(w_{2})\cap N_{G}(w_{3})|=4$ and $|N_{G}(\varphi_{4}(w_{1}))\cap N_{G}(\varphi_{4}(w_{2}))\cap N_{G}(\varphi_{4}(w_{3}))|=4$, it follows that there is a vertex $u\in\varphi_{5}(N_{G[S\cup\varphi_{4}(S)]}(z_{4}))$ with $|\nabla(u)\cap\nabla(H)|\geq4$, a contradiction.

Thus $\varphi_{4}(N_{G[S]}(z_{1}))\cap N_{G[S]}(z_{1})=\emptyset$. By Claim 1, it follows that $\varphi_{4}(\{z_{2},z_{3},z_{4}\})=N_{G}(w_{4})\backslash S$. Let $\varphi_{6}$ be an automorphism of $G$ such that $\varphi_{6}(z_{1})=z_{3}$. Similarly, we have
$\varphi_{6}(N_{G[S]}(z_{1}))\cap N_{G[S]}(z_{1})=\emptyset$ and $\varphi_{6}(\{z_{2},z_{3},z_{4}\})=N_{G}(w_{4})\backslash S$. It implies that $G[N_{G}(w_{4})\cup\varphi_{4}(N_{G[S]}(z_{1}))\cup\varphi_{6}(N_{G[S]}(z_{1}))$ has a subgraph isomorphic to $K_{3,5}$ or $G_{10}$, contradicting Claim 2 or Claim 3.

\textbf{Case 2.} $9\leq|S|\leq14$.

By Lemma \ref{5rec}(b), $\delta(G[S])\geq4$. If $|S|=9$, then
$18=\frac{1}{2}(6|S|-\lambda_{8}(G))=|E(S)|\geq\frac{1}{2}(|S|+1)\delta(G[S])\geq20$,
a contradiction. Thus $|S|\geq10$. If $|S|\leq13$, then let $W$ and
$Z$ be the bipartition of $G[S]$ with $|Z|\leq|W|$ and we have
$|W|=|Z|+\frac{1}{2}(1-(-1)^{|S|})$.

\textbf{Subcase 2.1.} Suppose first that $10\leq|S|\leq12$.

We claim that $d_{G[S]}(u)\leq5$ for each $u\in S$. Otherwise, suppose that there is a vertex $v_{1}\in S$ with $d_{G[S]}(v_{1})=6$. Choose a vertex $u_{9}\in X$ with $\nabla(u_{9})\cap\nabla(H)\neq\emptyset$. Let $\varphi_{7}$ be an automorphism of $G$ such that $\varphi_{7}(v_{1})=u_{9}$. As $\delta(G[S])\geq4$, it follows that $\varphi_{7}(S\backslash N_{G}(v_{1}))\subseteq X$, which implies that $|\nabla(u)\cap\nabla(V(H))|\geq4$ for each $u\in\varphi(N_{G}(v_{1}))\cap V(H)$, a contradiction.

Noting that $4\leq d_{G[S]}(u)\leq5$ for each $u\in S$, and recalling $|E(S)|=3|S|-9$ and $|W|=|Z|+\frac{1}{2}(1-(-1)^{|S|})$, it follows that there is a vertex $v_{2}\in Z$ and $v_{3}\in N_{G[S]}(v_{2})$ such that $d_{G[S]}(v_{2})=d_{G[S]}(v_{3})+1=5$.

Now we claim that each edge in $G$ is contained in a 4-cycle of $G$.
Otherwise, suppose that $G$ has an edge contained in no 4-cycles. It
follows by the vertex-transitivity of $G$ that each vertex in $G$ is
incident with an edge contained in no 4-cycles of $G$. Let
$\varphi_{8}$ be an automorphism of $G$ such that
$\varphi_{8}(v_{3})=v_{2}$. Then $\varphi_{8}(S)\neq S$. Note that
each edge in $G[S]$ is contained in a 4-cycle of $G[S]$. We have
$\varphi_{8}(N_{G[S]}(v_{3}))\subseteq N_{G[S]}(v_{2})$ and
$N_{G[S]}(\varphi_{8}(v_{2}))\subseteq\varphi_{8}(N_{G[S]}(v_{2}))$.
It implies $|\varphi_{8}(S)\cap S|\geq8$, contradicting Lemma
\ref{5rec}(c).

Thus $|\nabla(u)\cap\nabla(H)|+|\nabla(v)\cap\nabla(H)|\geq3$ for each edge $uv\in\nabla(H)$. Then there is a vertex $u_{10}\in V(G)$ with $|\nabla(u_{10})\cap\nabla(H)|\geq2$.

Suppose $|S|=10$. Then $|W|=|Z|=5$. Let $\varphi_{9}$ be an automorphism of $G$ such that $\varphi_{9}(v_{2})=u_{10}$. Then there is a vertex $u\in\varphi_{9}(N_{G[S]}(v_{2}))$ with $|\nabla(u)\cap\nabla(H)|\geq4$, a contradiction.

Thus $11\leq|S|\leq12$. Let $R_{i}$ be the set of vertices $u$ in $S$ with $d_{G[S]}(u)=i$ for $i=4,5$. Then $|R_{5}|=|R_{5}\cap Z|=4$ if $|S|=11$, and $|R_{5}\cap W|=|R_{5}\cap Z|=3$ if $|S|=12$.

Suppose that there is a vertex $u_{11}\in V(G)$ with $|\nabla(u_{11})\cap\nabla(H)|=3$. For a vertex $v\in S$, let $\psi$ be an automorphism of $G$ such that
$\psi(v)=u_{11}$. Then $\psi(S)\cap V(H)\neq\emptyset$ and $\psi(S)\backslash V(H)\neq\emptyset$. Noting that $\delta(G[S])\geq4$ and $|\nabla(u)\cap\nabla(H)|\leq3$ for each $u\in V(G)$, it follows that $|\psi(S)\cap X|=4$ and $G[\psi(S)\cap V(H)]$ and $G[\psi(S)\backslash V(H)]$ is isomorphic to $K_{1,4}$ or $K_{2,4}$. It implies
$|N_{G[S]}(v)\cap R_{4}|\geq\lfloor\frac{|S|}{6}\rfloor$ and that there are two vertices $v',v''\in R_{4}$ with $N_{G[S]}(v')=N_{G[S]}(v'')$. If $|S|=11$, then $N_{G[S]}(u)\cap R_{4}=\emptyset$ for each $u\in W\backslash N_{G[S]}(R_{4}\cap Z)$, a contradiction. Thus $|S|=12$. Then $|N_{G[S]}(u)\cap R_{4}|\geq2$ for each $u\in S$. So $\delta(G[R_{4}])\geq2$. Noting $|R_{4}|=|R_{5}|=6$, we have $12\geq4|R_{4}|-\delta(G[R_{4}])|R_{4}|\geq4|R_{4}|-2|E(R_{4})|=|[R_{4},R_{5}]|=5|R_{5}|
-2|E(R_{5})|\geq30-18$, which implies $d_{G[R_{4}]}(u)=2$ for each $u\in R_{4}$. Then $G[R_{4}]$ is a 6-cycle of $G$, which contradicts that $R_{4}$ has two vertices $v'$ and $v''$ with $N_{G[S]}(v')=N_{G[S]}(v'')$.

So $|\nabla(u)\cap\nabla(H)|\leq2$ for each
$u\in V(G)$. Then $|\nabla(u_{10})\cap\nabla(H)|=2$. We can see that there is no automorphism $\varphi$ of $G$ such that $\varphi(v_{2})=u_{10}$, contradicting that $G$ is vertex-transitive.

\textbf{Subcase 2.2.} Now suppose $13\leq|S|\leq14$.

\vskip 2mm
\textbf{Claim 4.} \emph{For two distinct $\lambda_{8}$-atoms $S_{1}$ and $S_{2}$ of $G$ with $S_{1}\cap S_{2}\neq\emptyset$, $G[S_{1}\backslash S_{2}]$ and $G[S_{1}\cap S_{2}]$ are isomorphic to $K_{3,3}$ or $K_{3,4}$.}
\vskip 2mm

By Lemma \ref{5rec}(c), we have $|S_{1}\backslash S_{2}|\leq7$,
$|S_{1}\cap S_{2}|\leq7$, $d_{G}(S_{1}\backslash
S_{2})+d_{G}(S_{2}\backslash S_{1})\leq2\lambda_{8}(G)$ and
$d_{G}(S_{1}\cap S_{2})+d_{G}(S_{1}\cup S_{2})\leq2\lambda_{8}(G)$.
Then $|S_{1}\backslash S_{2}|\geq6$ and $|S_{1}\cap S_{2}|\geq6$. By
Lemma \ref{5rec}(a), each of $d_{G}(S_{1}\backslash S_{2})$,
$d_{G}(S_{2}\backslash S_{1})$, $d_{G}(S_{1}\cap S_{2})$ and
$d_{G}(S_{1}\cup S_{2})$ is not less than $\lambda_{5}(G)$. Noting
$\lambda_{5}(G)=\lambda_{8}(G)=18$, we have $d_{G}(S_{1}\backslash
S_{2})=d_{G}(S_{1}\cap S_{2})=18$. Hence $G[S_{1}\backslash S_{2}]$
and $G[S_{1}\cap S_{2}]$ are isomorphic to $K_{3,3}$ or $K_{3,4}$.
So Claim 4 holds.

\vskip 2mm

Noting that $G[S]$ is not a regular graph, there is an automorphism $\varphi_{10}$ of $G$ such that $\varphi_{10}(S)\neq S$ and $\varphi_{10}(S)\cap S\neq\emptyset$. Then $G[S\backslash\varphi_{10}(S)]$ and $G[S\cap\varphi_{10}(S)]$ are isomorphic to $K_{3,3}$ or $K_{3,4}$ by Claim 4. Set $B=S\cap\varphi_{10}(S)$.

\vskip 2mm
\textbf{Claim 5.} \emph{$S$ has no subset $A$ different from $S\backslash B$ and $B$ such that $G[A]$ is isomorphic to $K_{3,4}$ and $G[S\backslash A]$ are isomorphic to $K_{3,3}$ or $K_{3,4}$.}
\vskip 2mm

Suppose to the contrary that $S$ has a subset $A$ satisfying the
above condition. Assume $|S|=13$. As $|W|=|Z|+1=7$, we konw $|A\cap
W|=4$. It follows that there is a vertex $v_{4}\in S$ with
$d_{G[S]}(v_{4})=6$. Choose a vertex $v_{5}\in S$ such that
$d_{G[S]}(v_{5})\geq5$ and $|\{v_{4},v_{5}\}\cap W|=1$. Let
$\varphi_{11}$ be an automorphism of $G$ such that
$\varphi_{11}(v_{5})=v_{4}$. Then $\varphi_{11}(S)\neq S$ and
$\varphi_{11}(N_{G[S]}(v_{5}))\subseteq N_{G[S]}(v_{4})$,
contradicting that $G[S\cap\varphi_{11}(S)]$ is isomorphic to
$K_{3,3}$ or $K_{3,4}$ by Claim 4. Assume next $|S|=14$. Then each
of $G[S\backslash B]$, $G[B]$, $G[A]$ and $G[S\backslash A]$ is
isomorphic to $K_{3,4}$. As $|E(S)|=33$, we know $d_{G[S]}(B)=9$. If
$|A\cap B|=1$, then $d_{G[S]}(B\backslash
A)=9=\frac{1}{2}d_{G}(B\backslash A)$, contradicting Lemma
\ref{5rec}(b). If $|A\cap B|=6$, then $d_{G[S]}(S\backslash(A\cup
B))=9=\frac{1}{2}d_{G}(S\backslash(A\cup B))$, contradicting Lemma
\ref{5rec}(b). If $2\leq|A\cap B|\leq5$, then $9=d_{G[S]}(B)\geq
d_{G[A]}(A\cap B)+d_{G[S\backslash A]}(B\backslash A)\geq5+5$, a
contradiction. Thus Claim 5 holds.

\vskip 2mm
\textbf{Claim 6.} \emph{Each vertex in $G$ is contained in exactly two distinct $\lambda_{8}$-atoms of $G$.}
\vskip 2mm

By the vertex-transitivity of $G$, it only needs to show that $S'=S$ or $\varphi_{10}(S)$ for a $\lambda_{8}$-atom $S'$ of $G$ with $S'\cap B\neq\emptyset$. Suppose $S'\neq S$ and $S'\neq\varphi_{10}(S)$. By Claims 4 and 5, we have $S'\cap S=B=S'\cap\varphi_{10}(S)$. Then $18=d_{G}(B)\geq d_{G[S]}(B)+d_{G[\varphi_{10}(S)]}(B)+d_{G[S']}(B)\geq3\times9$, a contradiction. Thus Claim 6 holds.

\vskip 2mm

Let $D$ be one of $S\backslash B$ and $B$ such that $G[D]$ is isomorphic to $K_{3,4}$. Choose two vertices $v_{6}$ and $v_{7}$ in $D$ such that $d_{G[D]}(v_{6})=d_{G[D]}(v_{7})-1=3$. By Claim 6, there is only one $\lambda_{8}$-atom $T$ of $G$ which is different from $S$ and contains $v_{6}$. By Claims 4 and 5, we have $S\cap T=D$. By Claim 6, $S$ and $T$ are also the only $\lambda_{8}$-atoms of $G$ which contain $v_{7}$. It implies that there is no automorphism $\varphi$ of $G$ such that $\varphi(v_{6})=v_{7}$, a contradiction.
\end{proof}

\section{Proof of Theorem \ref{main}}

If $G$ is 4-factor-critical, then by Theorem \ref{econnectivity} and Theorem \ref{pfc} we have $k=\lambda(G)\geq5$. So we consider the sufficiency. Suppose $k\geq5$. We will prove that $G$ is 4-factor-critical.

Suppose to the contrary that $G$ is not
4-factor-critical. We know by Theorem \ref{vtg} that $G$ is bicritical. By Lemma \ref{no4fc}, there is a subset
$X\subseteq V(G)$ with $|X|\geq4$ such that $c_{0}(G-X)=|X|-2$ and
every component of $G-X$ is factor-critical. Let $H_{1}$, $H_{2}$, $\dots$, $H_{p}$, $H_{p+1}$, $\dots$, $H_{t}$ be the components of $G-X$, where $t=|X|-2$ and $H_{1}$, $H_{2}$, $\dots$, $H_{p}$ are the nontrivial components of $G-X$. We know $p\geq1$ by Lemma \ref{component}. For each $i\in[p]$, since $H_{i}$ is factor-critical, $\delta(H_{i})\geq2$. For every subset $J\subseteq[t]$, we have
$$\sum_{i\in J}d_{G}(H_{i})+\lambda(G)(t-|J|)\leq\sum_{i=1}^{t}d_{G}(H_{i})\leq d_{G}(X)=k(t+2)-2|E(X)|,$$
which implies
\begin{align}
\sum_{i\in J}d_{G}(H_{i})+2|E(X)|\leq k(|J|+2).
\end{align}
Hence $|E(X)|\leq k$. Set $Y=\bigcup_{j=p+1}^{t}V(H_{j})$.

\textbf{Case 1.} $g(G)=3$.

By Lemma \ref{lessnumedge}, $|E(X)|\geq t-p=|X|-2-p$.

\textbf{Subcase 1.1.} Suppose that $d_{G}(A)\geq2k-2$ for all $A\subseteq V(G)$ with $2\leq|A|\leq|V(G)|-2$.

For each $i\in[p]$, we have $d_{G}(H_{i})\geq2k-2$. If $k$ is odd,
then $d_{G}(H_{i})$ is odd and hence $d_{G}(H_{i})\geq2k-1$. So
$d_{G}(H_{i})\geq2k-\frac{1}{2}(3+(-1)^{k})$ for each $i\in[p]$. Now
we have
\begin{align}
(2k-\frac{1}{2}(3+(-1)^{k}))p+2(|X|-2-p)\leq\sum_{i=1}^{p}d_{G}(H_{i})+2|E(X)|\leq k(p+2),
\end{align}
which implies $(k-2-\frac{1}{2}(3+(-1)^{k}))p+2(|X|-2-k)\leq0$. Hence $|X|\leq k+1$.

Suppose $|X|<k$. Then $p=t=|X|-2$. By Theorem
\ref{vtc}, $|X|\geq\kappa(G)>\frac{2}{3}k$. Hence we know from (2) that $2k\geq(k-\frac{1}{2}(3+(-1)^{k}))p>(k-\frac{1}{2}(3+(-1)^{k}))(\frac{2}{3}k-2)$. That is, $k^{2}-7k+3<0$ if $k$ is odd and $k^{2}-8k+6<0$ otherwise. It follows that $k\leq6$. If $k=6$, then $|X|\geq\kappa(G)=k$ by Lemma \ref{connectivity}, a contradiction. Thus $k=5$. Then $\kappa(G)=|X|=4$. By Lemma \ref{vtc2}, $\tau(G)=2$. It implies that there is an edge $x_{0}y_{0}\in E(G)$ such that $|N_{G}(x_{0})\cap N_{G}(y_{0})|=4$.

Noting $k=5$, we know from (2) that $|E(X)|\leq1$. Choose a vertex
$u\in X$ with $d_{G[X]}(u)=0$. Since $G$ is vertex-transitive, there
is an automorphism $\varphi_{1}$ of $G$ such that
$\varphi_{1}(x_{0})=u$. Assume $\varphi_{1}(y_{0})\in V(H_{1})$
without loss of generality. Noting $|N_{G}(x_{0})\cap
N_{G}(y_{0})|=4$, we have $N_{G}(u)\subseteq V(H_{1})$. Then
$d_{G}(V(H_{1})\cup\{u\})=d_{G}(X)-d_{G}(H_{2})-5\leq20-9-5<2k-2$, a
contradiction.

Thus $k\leq|X|\leq k+1$. Noting $(k-2-\frac{1}{2}(3+(-1)^{k}))p+2(|X|-2-k)\leq0$, we have
$p\leq2$ and $k\leq7$. Then $|Y|=|X|-2-p\geq k-4\geq1$. For any given vertex $v$, let $q$ be the number of triangles containing $v$ in $G$. By the vertex-transitivity of $G$, each vertex in $G$ is contained in $q$ triangles of $G$, which implies that each edge in $G$ is contained in at most $q$ triangles of $G$.

\vskip 2mm

\textbf{Claim 1.} \emph{$E(X)$ is a matching of $G$.}

\vskip 2mm

Assume $p=2$ or $|X|=k+1$. Then we know from (2) that $|E(X)|=|X|-2-p=|Y|$. Noting that there are $q|Y|$ triangles of $G$ containing one vertex in $Y$, it follows that each edge in $E(X)$ is contained in $q$ triangles of $G$, which implies that $E(X)$ is a matching of $G$. Next we assume $p=1$ and $|X|=k$. If two edges in $E(X)$ are adjacent, then $|E(X)|=q\geq2|Y|=2(k-3)$ and hence $d_{G}(H_{1})+2|E(X)|\geq2k-2+4(k-3)>3k$, which contradicts the inequality (1). So Claim 1 holds.

\vskip 2mm

By Claim 1, it follows that each edge incident with a vertex in $Y$ is contained in at most one triangle of $G$. Then, by the vertex-transitivity of $G$, each edge in $E(X)$ is contained in at most one triangle of $G$.

Suppose $|X|=k+1$. From (2), we know $k\leq6$, $p=1$ and $|E(X)|=|Y|=k-2$. Then each edge in $E(X)$ is contained in $q$ triangles of $G$. Noting that each edge in $E(X)$ is contained in at most one triangle of $G$, we have $q=1$. Then $|E(N_{G}(u))|=1$ for each $u\in Y$, which implies $|X|\geq2|E(X)|+(k-|E(X)|-1)=2k-3>k+1$, a contradiction.

Thus $|X|=k$. Then for each $e\in E(X)$ and each $u\in Y$, $G$ has a
triangle containing $e$ and $u$. As each edge in $E(X)$ is contained
in at most one triangle of $G$, it follows that $|Y|=1$, which
implies $p=2$ and $k=5$. From (2), we know
$d_{G}(H_{1})=d_{G}(H_{2})=9$ and $|E(X)|=1$. Assume
$|V(H_{1})|\leq|V(H_{2})|$. Let $u_{1}$ be the vertex in $Y$. For a
vertex $u_{2}\in V(H_{1})$ with $N_{G}(u_{2})\cap X\neq\emptyset$,
we have $|N_{G}(u_{2})\cap X|\leq3$ as $\delta(H_{1})\geq2$. As
$H_{2}$ is a component of $G-N_{G}(u_{1})$ with maximum cardinality,
it follows by the vertex-transitivity of $G$ that $H_{2}$ also is a
component of $G-N_{G}(u_{2})$ with maximum cardinality. Then
$N_{G}(X\backslash N_{G}(u_{2}))\subseteq V(H_{1})\cup Y$. Thus
$d_{G}(V(H_{1})\cup(X\backslash N_{G}(u_{2})))<8=2k-2$, a
contradiction. Hence Subcase 1.1 cannot occur.

\textbf{Subcase 1.2.} Suppose that there is a subset $A\subseteq V(G)$ with $2\leq|A|\leq|V(G)|-2$ such that $d_{G}(A)<2k-2$.

We choose a subset $S$ of $V(G)$ such that
$1<|S|\leq\frac{1}{2}|V(G)|$, $d(S)$ is as small as possible, and,
subject to these conditions, $|S|$ is as small as possible. Then
$d_{G}(S)\leq d_{G}(A)\leq2k-3$. By Corollary \ref{small2}, $d_{G}(S)=|S|\geq k$ and $G[S]$ is $(k-1)$-regular. As
$2k-3<\frac{2}{9}(k+1)^2$, $S$ is an imprimitive block of $G$ by
Theorem \ref{small}. Thus $G[S]$ is vertex-transitive by
Lemma \ref{block}. We also know that the orbits $S=S_{1}$, $S_{2}$, $\dots$, $S_{m_{1}}$ of $S$ under the automorphism group of $G$ form a partition of $V(G)$ and each $G[S_{i}]$ is $(k-1)$-regular.

Set $I_{i}=\{j\in\{1,2,\dots,m_{1}\}: S_{j}\cap V(H_{i})\neq\emptyset\}$ for each $i\in[t]$ and set $\mathscr{M}=\{\bigcup_{j\in I_{i}}S_{j}:i\in[t]\}$. If any two sets in $\mathscr{M}$ are disjoint, then $2|X|\geq2|\bigcup_{U\in\mathscr{M}}\nabla(U)|\geq\sum_{U\in\mathscr{M}}d_{G}(U)\geq
|\mathscr{M}|d_{G}(S)$.

Suppose $|S|=k$. Then each $G[S_{i}]$ is isomorphic to $K_{k}$ and hence it has common vertices with at most one component of $G-X$. Hence $|\mathscr{M}|=c_{0}(G-X)=|X|-2$ and any two sets in $\mathscr{M}$ are disjoint. Then $2|X|\geq
|\mathscr{M}|d_{G}(S)=(|X|-2)k>2|X|$, a contradiction.

Suppose $|S|=k+1$. As $\delta(H_{j})\geq2$ for each $j\in[p]$, we have that for each $S_{i}$, $|S_{i}\backslash X|=|S_{i}\cap Y|=2$ or $S_{i}\backslash X\subseteq V(H_{i'})$ for some $i'\in[t]$. Hence $|\mathscr{M}|\geq p+\frac{1}{2}(t-p)=\frac{1}{2}(t+p)\geq\frac{1}{2}(t+1)=\frac{1}{2}(|X|-1)$ and any two sets in $\mathscr{M}$ are disjoint. Then $2|X|\geq
|\mathscr{M}|d_{G}(S)\geq\frac{1}{2}(|X|-1)(k+1)>2|X|$, a contradiction.

Thus $|S|\geq k+2$. Noting that $(k-1)|S|$ is even and $k+2\leq|S|\leq2k-3$, we have $|S|=k+2$ if $5\leq k\leq6$. For each $i\in[p]$, if $V(H_{i})\cap S_{j}\neq\emptyset$, then $|V(H_{i})\cap S_{j}|\geq2$ as $\delta(H_{i})\geq2$.

\vskip 2mm
\textbf{Claim 2.} \emph{For each $S_{i}$, there is a element $a_{i}\in[p]$ such that $V(H_{a_{i}})\cap S_{i}\neq\emptyset$.}
\vskip 2mm

Suppose $S_{i}\subseteq X\cup Y$. By Lemma \ref{independent}, $|S_{i}\cap Y|\leq\frac{1}{3}|S_{i}|$. If $k\geq6$, then $|E(X)|\geq|E(S_{i}\cap X)|=\frac{1}{2}(k-1)(|S_{i}\cap X|-|S_{i}\cap Y|)\geq\frac{1}{6}(k-1)|S_{i}|\geq\frac{1}{6}(k-1)(k+2)>k$, a contradiction. Thus $k=5$. Then $|S|=k+2$ and $|S_{i}\cap Y|\leq\lfloor\frac{1}{3}|S_{i}|\rfloor=2$. Hence $|E(X)|\geq|E(S_{i}\cap X)|\geq\frac{1}{2}(k-1)(|S_{i}|-4)=\frac{1}{2}(k-1)(k-2)>k$, a contradiction. So Claim 2 holds.

\vskip 2mm
\textbf{Claim 3.} \emph{$X\backslash S_{i}\neq\emptyset$ for each $S_{i}$.}
\vskip 2mm

Suppose $X\subseteq S_{i}$. Choose a component $H_{j}$ of $G-X$ such
that $H_{j}\neq H_{a_{i}}$. Then $|V(H_{j})\cap
S_{i}|=|N_{G}(V(H_{j})\cap S_{i})\backslash
S_{i}|\leq|V(H_{j})\backslash S_{i}|$. Hence $V(H_{j})\backslash
S_{i}\neq\emptyset$. Then there is some $S_{i'}\subseteq
V(H_{j})\backslash S_{i}$. Now we know $d_{G}(V(H_{j})\backslash
S_{i})\geq d_{G}(S)=|S_{i}|$. On the other hand, we have
$d_{G}(V(H_{j})\backslash S_{i})\leq|S_{i}\backslash
V(H_{a_{i}})|<|S_{i}|$, a contradiction. So Claim 3 holds.

\vskip 2mm
\textbf{Claim 4.} \emph{For each $i\in[p]$, we have $d_{G}(H_{i})\geq2k-2$ if there is some $S_{j}$ such that $S_{j}\cap V(H_{i})\neq\emptyset$ and $S_{j}\backslash V(H_{i})\neq\emptyset$.}
\vskip 2mm

Suppose $S_{j}\cap V(H_{i})\neq\emptyset$ and $S_{j}\backslash
V(H_{i})\neq\emptyset$. By Claim 3, $X\backslash
S_{j}\neq\emptyset$. Suppose $|\overline{V(H_{i})\cup S_{j}}|=1$.
Then $\overline{V(H_{i})\cup S_{j}}=X\backslash S_{j}$, which
implies $|\overline{V(H_{i})\cup X}|=1$. Hence $t=2$ and $p=1$,
implying $t=|X|-2\geq k-2>2$, a contradiction. Thus
$|\overline{V(H_{i})\cup S_{j}}|\geq2$. Then $|S_{j}|=d_{G}(S)\leq
d_{G}(V(H_{i})\cup S_{j})\leq|[V(H_{i}),\overline{V(H_{i})\cup
S_{j}}]|+|S_{j}\backslash V(H_{i})|$, which implies
$|[V(H_{i}),\overline{V(H_{i})\cup S_{j}}]|\geq|S_{j}\cap
V(H_{i})|$. Hence $d_{G}(H_{i})\geq d_{G[S_{j}]}(S_{j}\cap
V(H_{i}))+|[V(H_{i}),\overline{V(H_{i})}\cap\overline{S_{j}}]|\geq
d_{G[S_{j}]}(S_{j}\cap V(H_{i}))+|S_{j}\cap V(H_{i})|$. If
$|S_{j}\backslash V(H_{i})|\geq2$, then $d_{G[S_{j}]}(S_{j}\cap
V(H_{i}))\geq2k-4$ by Corollary \ref{2subset}, which implies
$d_{G}(H_{i})\geq2k-4+|S_{j}\cap V(H_{i})|\geq2k-2$. If
$|S_{j}\backslash V(H_{i})|=1$, then $d_{G}(H_{i})\geq
k-1+|S_{j}\cap V(H_{i})|\geq2k$. Claim 4 holds.

\vskip 2mm
\textbf{Claim 5.} \emph{$S_{i}\subseteq V(H_{a_{i}})\cup X$ for each $S_{i}$.}
\vskip 2mm

Suppose, to the contrary, that $G-X$ has a component $H_{b}$ with $V(H_{b})\cap(S_{i}\backslash V(H_{a_{i}}))\neq\emptyset$. Let $\theta$ be an integer such that $\theta=1$ if $|V(H_{b})=1$ and $\theta=0$ otherwise. As $X\backslash S_{i}\neq\emptyset$ by Claim 2, there is some $S_{j}$ with $S_{j}\cap(X\backslash S_{i})\neq\emptyset$. Set $J=\{a_{i},b\}\cup\{a_{j}\}$. For each $i'\in[p]$, we have $d_{G}(H_{i'})\geq d_{G}(S)\geq k+2$ and furthermore $d_{G}(H_{i'})\geq2k-2$ by Claim 4 if $i'\in[p]\cap J$. If $|J|=2$, then, noting that $d_{G[S_{i}]}(V(H_{a_{j}})\cap S_{i})\geq2k-4$ by Corollary \ref{2subset} and $\lambda(G[S_{j}])=k-1$ by Theorem \ref{econnectivity}, we have $d_{G}(H_{a_{j}})\geq d_{G[S_{i}]}(V(H_{a_{j}})\cap S_{i})+d_{G[S_{j}]}(V(H_{a_{j}})\cap S_{j})\geq2k-4+k-1=3k-5$.

Assume $5\leq k\leq 6$. We know that $|S|=k+2$ and $G[S_{i}\cap V(H_{i'})]$ is isomorphic to $K_{2}$ for each $i'\in\{a_{i},b\}\cap[p]$. Hence $S_{i}\subseteq V(H_{a_{j}})\cup V(H_{b})\cup X$. If $\theta=1$, then $|E(G[S_{i}\cap X])|=\frac{1}{2}((k-1)|S_{i}\cap X|-(k-1)-(2k-4))=\frac{1}{2}(k^{2}-5k+6)\geq3$. If $\theta=0$, then $k=6$ as $G[S_{i}]$ is vertex-transitive, which implies $|E(S_{i}\cap X)|=2$. Now we have
\begin{align*}
&\sum_{i'\in J}d_{G}(H_{i'})+2|E(X)|\\
\geq&(3k-5)(3-|J|)+2(2k-2)(|J|-2)+\theta k+(1-\theta)(2k-2)+2|E(S_{i}\cap X)|\\
=&k(|J|+2)+|J|+k-9-\theta(k-2)+2|E(S_{i}\cap X)|>k(|J|+2),
\end{align*}
which contradicts the inequality (1).

Assume $k\geq7$. If $\theta=1$, then $t=|X|-2\geq k-2\geq5$. If $\theta=0$, then $t=|X|-2\geq\lceil\frac{2k}{3}\rceil-2\geq3$ by Theorem \ref{vtc}. Now we have
\begin{align*}
&\sum_{i'\in[t]}d_{G}(H_{i'})+2|E(X)|\\
\geq&(3k-5)(3-|J|)+2(2k-2)(|J|-2)+\theta(p-|J|+1)(k+2)+\\
&(1-\theta)(2k-2+(p-|J|)(k+2))
+(t-p)k+2(t-p)\\
=&k(t+2)+2t+\theta(k+2)+(1-\theta)(2k-2)-|J|-k-7>k(t+2),
\end{align*}
which contradicts the inequality (1). So Claim 5 holds.

\vskip 2mm

By Claims 2 and 5, it follows that $|\mathscr{M}|=p=t$ and any two sets in $\mathscr{M}$ are disjoint. Then $2|X|\geq|\mathscr{M}|d_{G}(S)\geq(|X|-2)(k+2)>2|X|$, a contradiction.

\textbf{Case 2.} $g(G)\geq4$.

For each $j\in[p]$, we know from (1) that $d_{G}(H_{j})\leq3k$. Let $F_{j}$ be a component of $G[\overline{V(H_{j})}]$ which contains a vertex in $V(G)\backslash(V(H_{j})$ $\cup X)$. Then $\nabla(F_{j})$ is a 5-restricted edge-cut of $G$. Hence $\lambda_{5}(G)\leq d_{G}(F_{j})\leq d_{G}(H_{j})\leq3k$. As it follows by Corollary \ref{2subset} that $\lambda_{4}(G)\geq2k-2$, we have $2k-2\leq\lambda_{4}(G)\leq\lambda_{5}(G)\leq3k$.

\vskip 2mm

\textbf{Claim 6.} \emph{If $\lambda_{5}(G)\geq4k-8$ and $k\leq6$, then $p=1$, $|V(H_{1})|\geq7$, $\lambda_{7}(G)\leq3k$ and furthermore, $\lambda_{8}(G)\leq3k$ if $\lambda_{5}(G)\geq4k-8$.}

\vskip 2mm

Suppose $\lambda_{5}(G)\geq4k-8$ and $k\leq6$. Then $p=1$ by Lemma \ref{component}. We claim that $G[\overline{V(H_{1})}]$ is connected. Otherwise, $d_{G}(H_{1})\geq\lambda(G)+d_{G}(F_{1})\geq k+\lambda_{5}(G)>3k$, a contradiction. Suppose $|V(H_{1})|=5$. As $g(G)\geq4$, $H_{1}$ is a 5-cycle of $G$. It follows that $k=5$, $E(X)=\emptyset$ and $|X|\geq8$. Then $g_{0}(G)\geq7$ by Lemma \ref{girthX}, a contradiction. Thus $|V(H_{1})|\geq7$. Then $\nabla(H_{1})$ is a 7-restricted edge-cut of $G$ and $\lambda_{7}(G)\leq d_{G}(V(H_{1}))\leq3k$. If $\lambda_{5}(G)>4k-8$, then $|X|\geq7$ and $|V(H_{1})|\geq9$ by Lemma \ref{component}, which implies $\lambda_{8}(G)\leq d_{G}(H_{1})\leq3k$. So Claim 6 holds.

\vskip 2mm

By Claim 6, we can discuss Case 2 in the following two subcases.

\textbf{Subcase 2.1.} Suppose that $k=5$, $\lambda_{5}(G)=12$ and $\lambda_{7}(G)\geq13$.

We have $\lambda_{4}(G)=12$. As $\lambda_{7}(G)$ exists,
$|V(G)|\geq14$. Then, by Lemma \ref{5rec}(a),
$d_{G}(A)\geq\lambda_{7}(G)$ for each subset $A\subseteq V(G)$ with
$|A|=7$, which implies that $G$ has no subgraphs isomorphic to
$K_{3,4}$. It follows by the vertex-transitivity of $G$ that $G$ has
no subgraphs isomorphic to $K_{2,5}$. By Claim 6, $p=1$ and
$|V(H_{1})|\geq7$. Hence $|X|\geq6$ and $|V(G)|\geq16$. By Lemma
\ref{5rec}(a), $d_{G}(V(H_{1})\cup A)\geq\lambda_{7}(G)$ for each
subset $A\subseteq X$ with $|A|\leq1$, which implies
$d_{G}(H_{1})\geq13$ and $|N_{G}(u)\cap V(H_{1})|\leq3$ for each
$u\in X$. Noting $\delta(H_{1})\geq2$, we have
$|\nabla(u)\cap\nabla(H_{1})|\leq3$ for each $u\in V(G)$.

\vskip 2mm
\textbf{Claim 7.} \emph{There is no subset $A\subseteq V(G)$ with $|A|\leq3$ such that $A\cap V(H_{1})\neq\emptyset$, $|\nabla(A)\cap\nabla(H_{1})|=3|A|$ and $d_{G}((V(H_{1})\cup A)\backslash(V(H_{1})\cap A))\leq12$.}
\vskip 2mm

Suppose to the contrary that such subset $A$ of $V(G)$ exists. Set
$B=(V(H_{1})\cup A)\backslash(V(H_{1})\cap A)$. Then $|B|\geq4$ and
$|\overline{B}|\geq7$. By Lemma \ref{5rec}(a), we have
$d_{G}(B)\geq\lambda_{4}(G)$ and furthermore,
$d_{G}(B)\geq\lambda_{7}(G)$ if $|B|\geq7$. As $d_{G}(B)\leq12$, we
know $|B|\leq6$ and $d_{G}(B)=12$. It implies that $E(V(H_{1})\cap
A)=\emptyset$ and $G[B]$ is isomorphic to $k_{2,2}$ or $K_{3,3}$.
Hence $G[V(H_{1})\cup A]$ is bipartite. Then $H_{1}$ is bipartite,
contradicting that $H_{1}$ is factor-critical. So Claim 7 holds.

\vskip 2mm

As $\lambda_{5}(G)=12<\lambda_{7}(G)$ and $k=5$, each $\lambda_{5}$-atom of $G$ induces a subgraph isomorphic to $K_{3,3}$. Let $T_{1}$, $T_{2}$, $\dots$, $T_{m_{2}}$ be all the subsets of $V(G)$, which induce subgraphs isomorphic to $K_{3,3}$. Let $R_{i}$ be the set of vertices in $X$ with $i$ neighbors in $V(H_{1})$ for $1\leq i\leq3$ and let $Q$ be the set of vertices in $V(H_{1})$ with 3 neighbors in $X$.

\textbf{Subcase 2.1.1.} \emph{Suppose that there are two distinct $T_{i}$ and $T_{j}$ with $T_{i}\cap T_{j}\neq\emptyset$.}

Noting that $G$ has no subgraphs isomorphic to $K_{3,4}$ or
$K_{2,5}$, we have $|T_{i}\cap T_{j}|=2$ or 4. If $|T_{i}\cap
T_{j}|=4$, then $d_{G}(T_{i}\cap T_{j})\leq12<\lambda_{7}(G)$, which
contradicts Lemma \ref{5rec}(a). Thus $|T_{i}\cap T_{j}|=2$. Assume
$T_{i}\cap T_{j}=\{v_{1},v_{2}\}$.

\vskip 2mm

\textbf{Claim 8.} \emph{For each $u\in X$ with $d_{G[X]}(u)=0$ and $N_{G}(u)\cap V(H_{1})\neq\emptyset$, we have $N_{G}(u)\cap V(H_{1})\subseteq Q$ if $u\in R_{1}\cup R_{2}$, and $|N_{G}(u)\cap V(H_{1})\cap Q|\geq1$ if $u\in R_{3}$.}

\vskip 2mm

Since $G$ is vertex-transitive, there is an automorphism $\varphi_{2}$ of $G$ such that $\varphi_{2}(v_{1})=u$. If $u\in R_{1}\cup R_{2}$, then $\varphi_{2}(N_{G}(v_{2}))\subseteq X$, which implies $N_{G}(u)\cap V(H_{1})\subseteq Q$. If $u\in R_{3}$, then $|\varphi_{2}(N_{G}(v_{2}))\cap X|\geq3$, which implies $|N_{G}(u)\cap V(H_{1})\cap Q|\geq1$. So Claim 8 holds.

\vskip 2mm

Assume $E(X)\neq\emptyset$. Then $|E(X)|=1$ and
$\sum_{i=1}^{3}i|R_{i}|=d_{G}(H_{1})=13$, which implies
$\sum_{i=1}^{3}|R_{i}|\geq5$. By Claim 8, $Q\neq\emptyset$. We have
$d_{G}(V(H_{1})\backslash\{u\})\leq12$ for each $u\in Q$,
contradicting Claim 7.

Thus $E(X)=\emptyset$. As $d_{G}(V(H_{1})\cup A)\geq\lambda_{4}(G)$
for each subset $A\subseteq X$ with $|A|=4$ by Lemma \ref{5rec}(a),
we have $|R_{3}|\leq3$. By Claim 6,
$|\nabla(Q)\cap\nabla(H_{1})|\geq|R_{3}|+2|R_{2}|+|R_{1}|=15-2|R_{3}|\geq9$,
which implies $|Q|\geq3$. Choose a subset $Q'\subseteq Q$ with
$|Q'|=3$. Then $d_{G}(V(H_{1})\backslash Q')\leq12$, contradicting
Claim 7. Hence Subcase 2.1.1 cannot occur.

\textbf{Subcase 2.1.2.} So we suppose that any two distinct $T_{i}$ and $T_{j}$ are disjoint.

By the vertex-transitivity of $G$, each vertex in $G$ is contained in a $\lambda_{5}$-atom of $G$. Hence $T_{1}$, $T_{2}$, $\dots$, $T_{m_{2}}$ form a partition of $V(G)$.

Assume $E(X)\neq\emptyset$. Noting $c_{0}(G-X)=|X|-2$ and
$|E(X)|=1$, it follows that there is some $T_{i}$ such that
$T_{i}\cap X\neq\emptyset$, $T_{i}\cap V(H_{1})\neq\emptyset$ and
$E(T_{i})\cap E(X)=\emptyset$. Then there is a vertex $u_{1}\in
T_{i}\cap(R_{3}\cup Q)$. By Claim 7, it follows that $u_{1}\in X$.
We have $d_{G}(V(H_{1})\cup\{u_{1}\})=12<\lambda_{7}(G)$,
contradicting Lemma \ref{5rec}(a).

Thus $E(X)=\emptyset$. Set $\mathscr{B}_{1}=\{T_{j}:|T_{j}\cap X|=3,
j\in[m_{2}]\}$ and $\mathscr{B}_{2}=\{T_{j}:|T_{j}\cap X|<3,
j\in[m_{2}]\}$. Let $D=(\bigcup_{A\in\mathscr{B}_{1}}A\cap
V(H_{1}))\cup(\bigcup_{A\in\mathscr{B}_{2}}A\cap X)$. Noting
$c_{0}(G-X)=|X|-2$ and $p=1$, we have $|D|=3$. By Claim 7, it
follows that $T\subseteq X$. If $|X|\geq7$, then
$d_{G}(H_{1}+D)=12<\lambda_{7}(G)$, which contradicts Lemma
\ref{5rec}(a). Thus $|X|=6$. As $G$ has no subgraphs isomorphic to
$K_{2,5}$, we know that $|R_{2}|=|R_{3}|=3$ and $G[Y\cup R_{3}]$ is
isomorphic to $K_{3,3}$. Choose a vertex $u_{2}\in R_{2}$ and a
vertex $u_{3}\in Y$. Let $\varphi_{3}$ be an automorphism of $G$
such that $\varphi_{3}(u_{3})=u_{2}$. Then $\varphi_{3}(Y)=R_{2}$
and $\varphi_{3}(R_{3})\subseteq V(H_{1})$. It implies $D\subseteq
V(H_{1})$ by the choice of $D$, a contradiction. Hence Subcase 2.1
cannot occur.

\textbf{Subcase 2.2.} Now we suppose that $k\neq5$, $\lambda_{5}(G)\neq12$ or $\lambda_{5}(G)=\lambda_{7}(G)=12$.

Let $S'$ be a $\lambda_{s}$-atom of $G$, where
$$ s=\left\{
\begin{aligned}
&4, \ \textrm{if}\ k\leq6\ \textrm{and}\ \lambda_{5}(G)<4k-8;\\
&7, \ \textrm{if}\ k=5\ \textrm{and}\ \lambda_{5}(G)=\lambda_{7}(G)=12; \\
&6, \ \textrm{if}\ k=5\ \textrm{and}\ \lambda_{5}(G)=\lambda_{6}(G)=13;\\
&7, \ \textrm{if}\ k=5, \lambda_{5}(G)=13\ \textrm{and}\ \lambda_{6}(G)=\lambda_{7}(G)=14;\\
&8, \ \textrm{if}\ k=5, \lambda_{5}(G)=13, \lambda_{6}(G)\geq14\ \textrm{and}\ \lambda_{8}(G)=15;\\
&5, \ \textrm{if}\ k=5\ \textrm{and}\ \lambda_{5}(G)=14;\\
&6, \ \textrm{if}\ k=5\ \textrm{and}\ \lambda_{5}(G)=\lambda_{6}(G)=15;\\
&5, \ \textrm{if}\ k=6\ \textrm{and}\ \lambda_{5}(G)=4k-8;\\
&8, \ \textrm{if}\ k=6\ \textrm{and}\ \lambda_{5}(G)=18;\\
&5, \ \textrm{if}\ k\geq7.
\end{aligned}
\right.
$$

\textbf{Claim 9.} \emph{$S'$ is an imprimitive block of $G$ such that $|S'|>\frac{1}{2}\lambda_{s}(G)$ if $k\leq6$ and $|S'|>\frac{1}{3}\lambda_{s}(G)$ otherwise.}

\vskip 2mm

Clearly, it holds by Lemma \ref{6to7rec} if $k=5$ and
$\lambda_{5}(G)=\lambda_{7}(G)=12$. So we assume $k>5$ or
$\lambda_{5}(G)\neq12$. By Lemma \ref{triangle},
$\frac{1}{2}|S'|^{2}\geq2|E(S')|=k|S|-\lambda_{s}(G)$. If $5\leq
k\leq6$ and $\lambda_{5}(G)<4k-8$, then $\frac{1}{2}|S'|^{2}\geq
k|S'|-\lambda_{s}(G)>k|S'|-4k+8$, which implies
$|S'|>2k-4\geq\textrm{max}\{2(s-1),\frac{1}{2}\lambda_{s}(G)\}$. If
$5\leq k\leq6$ and $\lambda_{5}(G)\geq4k-8$, then $|S'|>2(s-1)$ and
$2|S'|>\lambda_{s}(G)$ by Lemmas \ref{5rec14to15} and
\ref{5rec13}-\ref{5rec18}. If $k\geq7$, then
$\frac{1}{2}|S'|^{2}\geq k|S'|-\lambda_{s}(G)\geq k|S'|-3k$ and
hence $|S'|>k+2>\textrm{max}\{2(s-1),\frac{1}{3}\lambda_{s}(G)\}$.
Suppose $S'$ is not an imprimitive block of $G$. Then there is an
automorphism $\varphi$ of $G$ such that $\varphi(S')\neq S'$ and
$\varphi(S')\cap S'\neq\emptyset$. By Lemma \ref{5rec}(c),
$|S'|=|S'\cap\varphi(S')|+|S'\backslash\varphi(S')|\leq2(s-1)$, a
contradiction. So Claim 9 holds.

\vskip 2mm

By Claim 9 and Lemma \ref{block}, $G[S']$ is vertex-transitive and hence it is $(k-1)$-regular if $k\leq6$ and is $(k-1)$-regular or $(k-2)$-regular otherwise. From Claim 9, we also know that the orbits $S'=S'_{1}$, $S'_{2}$, $\dots$, $S'_{m_{3}}$ of $S'$ under the automorphism group of $G$ form a partition of $V(G)$.

\vskip 2mm

\textbf{Claim 10.} \emph{$G[S']$ is $(k-1)$-regular.}

\vskip 2mm

Suppose that $G[S']$ is $(k-2)$-regular. Then $k\geq7$, $s=5$ and
$2|S'|=\lambda_{s}(G)\leq3k$, which implies $|S'|\leq\frac{3}{2}k$.
By Lemma \ref{triangle},
$\frac{1}{4}|S'|^{2}\geq|E(S')|=\frac{1}{2}(k-2)|S'|$, which implies
$|S'|\geq2(k-2)$. Now $2(k-2)\leq|S'|\leq\frac{3}{2}k$, which
implies $k\leq8$ and $|S'|=2(k-2)$. Hence $G[S']$ is isomorphic to
$K_{k-2,k-2}$. For each $i\in[p]$, noting $3k\geq
d_{G}(H_{i})\geq\lambda_{s}(G)=4(k-2)$ and that $d_{G}(H_{i})$ has
the same parity with $k$, we have $d_{G}(H_{i})=3k$. Hence $p=1$,
$E(X)=\emptyset$, $|V(H_{1})|>5$ and $|X|\geq k$. Noting that
$c_{0}(G-X)=|X|-2$, there is some $S'_{i}$ with $S'_{i}\cap
X\neq\emptyset$ and $S'_{i}\cap V(H_{1})\neq\emptyset$. Then there
is a vertex $u\in S'_{i}$ with $|\nabla(u)\cap\nabla(H_{1})|\geq
k-2$. Then $d_{G}(V(H_{1})\cup\{u\})\leq
d_{G}(H_{1})-(k-4)=2k+4<4(k-2)=\lambda_{s}(G)$ if $u\in X$ and
$d_{G}(V(H_{1})\backslash\{u\})<\lambda_{s}(G)$ otherwise,
contradicting Lemma \ref{5rec}(a). So Claim 10 holds.

\vskip 2mm

As $\delta(H_{i})\geq2$ for each $i\in[p]$, it follows by Claim 10 that $\delta(G[V(H_{j})\cap S'_{i}])\geq1$ if $V(H_{j})\cap S'_{i}\neq\emptyset$.

\vskip 2mm
\textbf{Claim 11.} \emph{For each $S'_{i}$, $S'_{i}\backslash(X\cup Y)\neq\emptyset$ or $|S'_{i}\cap X|=|S'_{i}\cap Y|$.}
\vskip 2mm

Suppose $|S'_{i}\cap X|>|S'_{i}\cap Y|$ for some $S'_{i}\subseteq X\cup Y$. If $G[S'_{i}]$ is bipartite, then $|S'_{i}\cap Y|\leq|S'_{i}\cap X|-2$. If $G[S'_{i}]$ is non-bipartite, then $|S'_{i}\cap Y|\leq\alpha(G[S'_{i}])\leq\frac{1}{2}|S'_{i}|-\frac{k-1}{4}$ by Lemma \ref{independent}, which implies $|S'_{i}\cap Y|\leq|S'_{i}\cap X|-\frac{k-1}{2}\leq|S_{i}\cap X|-2$. Thus $|E(S'_{i}\cap X])|=\frac{1}{2}(k-1)(|S'_{i}\cap X|-|S'_{i}\cap Y|)\geq k-1$. Noting $d_{G}(H_{1})\geq\lambda_{5}(G)\geq2k-2$, we have
$d_{G}(H_{1})+2|E(X)|\geq2k-2+2(k-1)>3k$, a contradiction. So Claim 11 holds.

\vskip 2mm

\textbf{Subcase 2.2.1.} Suppose $|S'|\leq2k-1$.

\vskip 2mm
\textbf{Claim 12.} \emph{If $S'_{i}\cap V(H_{j})\neq\emptyset$ for some $j\in[p]$, then $S'_{i}\subseteq V(H_{j})\cup X$.}
\vskip 2mm

Suppose $S'_{i}\cap V(H_{j})\neq\emptyset$ for some $j\in[p]$ and $S'_{i}\cap V(H_{j'})\neq\emptyset$ for some $j'\in[t]\backslash\{j\}$. As $\delta(G[S'_{i}\cap V(H_{j})])\geq1$, there is an edge $x_{1}y_{1}\in E(S'_{i}\cap V(H_{j}))$. Then $|S'_{i}\cap(V(H_{j})\cup X)|\geq|N_{G[S'_{i}]}(x_{1})\cup N_{G[S'_{i}]}(y_{1})|=2k-2$. It implies $|S'_{i}\cap V(H_{j'})|=1$ and $|S'_{i}|=2k-1$. Then $|V(H_{j'})|=1$ and $|X|\geq |N_{G}(V(H_{j'}))|=k$. Hence $|\overline{V(H_{j})\cup S'_{i}}|\geq|N_{G}(V(H_{j'}))\backslash S'_{i}|+(c_{0}(G-X)-2)\geq1+k-4\geq2$. By Corollary \ref{2subset}, we have
\begin{align*}
2k-2&\leq d_{G}(V(H_{j})\cup S'_{i})\\
&\leq d_{G}(H_{j})-d_{G[S'_{i}]}(S'_{i}\cap V(H_{j}))+|S'_{i}\backslash V(H_{j})|\\
&=d_{G}(H_{j})-((k-1)|S'_{i}\cap X|-2|E(S'_{i}\cap X)|-(k-1))+|S'_{i}\cap X|+1\\
&=d_{G}(H_{j})+2|E(S'_{i}\cap X)|-(k-2)|S'_{i}\cap X|+k\\
&\leq 3k-(k-2)(k-1)+k=-k^{2}+7k-2,
\end{align*}
which implies $k=5$. It is easy to verify that there is no triangle-free non-bipartite 4-regular graph of order 9, which implies $|S'|\neq9=2k-1$, a contradiction. So Claim 12 holds.

\vskip 2mm

Set $I'_{i}=\{j\in[m_{3}]: S'_{j}\cap V(H_{i})\neq\emptyset\}$ for
each $i\in[t]$ and $\mathscr{M}'=\{\bigcup_{j\in
I'_{i}}S'_{j}:i\in[t]\}$. Then any two sets in $\mathscr{M}'$ are
disjoint by Claim 12. By Lemma \ref{5rec}(a),
$d_{G}(U)\geq\lambda_{s}(G)$ for each $U\in \mathscr{M}'$. Then, by
Claim 11, we have
\begin{align*}
&2(p+2+(k-1)(|\mathscr{M}'|-p))\\
=&2|X|\geq2|\bigcup_{U\in\mathscr{M}'}\nabla(U)|=
\sum_{U\in\mathscr{M}'}d_{G}(U)\geq|\mathscr{M}'|\lambda_{s}(G)\geq|\mathscr{M}'|(2k-2),
\end{align*}
which implies $p\leq\frac{2}{k-2}<1$, a contradiction. Hence Subcase 2.2.1 cannot occur.

\textbf{Subcase 2.2.2} So we suppose $|S'|\geq2k$.

We have $\lambda_{s}(G)=|S'|\geq2k$. If $s=4$, then $\lambda_{5}(G)\geq\lambda_{s}(G)\geq2k$. If $s\geq5$, then $\lambda_{5}(G)\geq2k$ by the choice of $s$. Then $2kp\leq p\lambda_{5}(G)\leq\sum_{i=1}^{p}d_{G}(H_{i})+2|E(X)|\leq k(2+p)$, which implies
$p\leq2$.

Let
$$\mathscr{N}=\{S'_{i}: S'_{i}\cap X\neq\emptyset\ \textrm{and}\ S'_{i}\backslash(X\cup Y)\neq\emptyset, i\in[m_{3}]\}.$$
By Claim 11, $\sum_{A\in\mathscr{N}}(|A\cap X|-|A\cap Y|)=\sum_{i=1}^{m_{3}}(|S'_{i}\cap X|-|S'_{i}\cap Y|)=|X|-|Y|=p+2$. Noting $|A\cap X|>|A\cap Y|$ for each $A\in\mathscr{N}$, we have $1\leq|\mathscr{N}|\leq p+2$. Choose a set $S'_{j_{1}}\in\mathscr{N}$. Without loss of generality, we assume $S'_{j_{1}}\cap V(H_{1})\neq\emptyset$.

Suppose $p=2$. Then $E(X)=\emptyset$ and
$2k=\lambda_{5}(G)=d_{G}(H_{1})=d_{G}(H_{2})$. Hence
$\lambda_{4}(G)=\lambda_{5}(G)=2k=|S'|$. For each $u\in V(G)$ and
each $i\in[p]$, we have $d_{G}(V(H_{i})\cup\{u\})\geq\lambda_{4}(G)$
and $d_{G}(V(H_{i})\backslash\{u\})\geq\lambda_{4}(G)$ by Lemma
\ref{5rec}(a), which implies $|\nabla(u)\cap\nabla(H_{i})|\leq k-3$.
Hence $|S'_{j_{1}}\backslash V(H_{1})|\geq2$ and
$\delta(G[S'_{j_{1}}\cap V(H_{1})])\geq2$, which implies
$|S'_{j_{1}}\cap V(H_{1})|\geq4$. Choose an edge $x_{2}y_{2}\in
E(S'_{j_{1}}\cap V(H_{1}))$. Then
$|S'_{j_{1}}\backslash(V(H_{1})\cup
X)|\leq|S'_{j_{1}}\backslash(N_{G[S'_{j_{1}}]}(x_{2})\cup
N_{G[S'_{j_{1}}]}(y_{2}))|=2$. It follows that $S'_{j_{1}}\cap
V(H_{2})=\emptyset$. Noting that $d_{G}(S'_{j_{1}}\cap
V(H_{1}))\geq2k-4$ by Corollary \ref{2subset}, we have
$|S'_{j_{1}}\cap X|\geq|S'_{j_{1}}\cap Y|+2$. Now
\begin{align*}
d_{G}(V(H_{1})\cup S'_{j_{1}})&\leq d_{G}(H_{1})-d_{G[S'_{j_{1}}]}(V(H_{1})\cap S'_{j_{1}})+|S'_{j_{1}}\backslash V(H_{1})|\\
&=2k-(k-1)(|S'_{j_{1}}\cap X|-|S'_{j_{1}}\cap Y|)+|S'_{j_{1}}\backslash V(H_{1})|\\
&\leq2k-2(k-1)+2k-4<2k=\lambda_{4}(G),
\end{align*}
contradicting Lemma \ref{5rec}(a).

Thus $p=1$. Suppose $|\mathscr{N}|=1$. Then $|S'_{j_{1}}\cap
X|=|S'_{j_{1}}\cap Y|+3$ and there is some
$S'_{j}\subseteq\overline{V(H_{1})\cup S'_{j_{1}}}$. We know by
Claim 11 that $G[S']$ is bipartite. Hence there is some
$S'_{j'}\subseteq V(H_{1})\backslash S'_{j_{1}}$. By Lemma
\ref{5rec}(a), we have
\begin{align*}
|S'|=\lambda_{s}(G)\leq d_{G}(V(H_{1})\cup S'_{j_{1}})&\leq d_{G}(H_{1})-d_{G[S'_{j_{1}}]}(S'_{j_{1}}\backslash V(H_{1}))+|S'_{j_{1}}\backslash V(H_{1})|\\
&=d_{G}(H_{1})+2|E(S'_{j_{1}}\cap X)|-3(k-1)+|S'_{j_{1}}\backslash
V(H_{1})|\\& \leq3k-3(k-1)+|S'_{j_{1}}\backslash V(H_{1})|.
\end{align*}
Similarly, we can obtain $|S'|\leq d_{G}(H_{1}-S'_{j_{1}})\leq3+|S'_{j_{1}}\cap V(H_{1})|$. Then $2|S'|\leq 6+|S'_{j_{1}}|$, which implies $|S'|\leq6<2k$, a contradiction.

Thus $|\mathscr{N}|\geq2$. For each $S'_{i}\in \mathscr{N}$, noting
$|S'_{i}\cap V(H_{1})|\geq2$, we have $d_{G[S'_{i}]}(S'_{i}\cap
V(H_{1}))\geq2k-4$ by Corollary \ref{2subset} if $|S'_{i}\backslash
V(H_{1})|\geq2$, which implies that $|S'_{i}\cap X|=1$ if
$|S'_{i}\cap X|=|S'_{i}\cap Y|+1$. If $|\mathscr{N}|=3$, then
$|S'_{i}\cap X|=1$ for each $S'_{i}\in\mathscr{N}$ and hence
$d_{G}(V(H_{1})\cup(\bigcup_{S'_{i}\in\mathscr{N}}S'_{i}))\leq
d_{G}(H_{1})-3(k-2)\leq6<\lambda_{s}(G)$, which contradicts Lemma
\ref{5rec}(a). Thus $|\mathscr{N}|=2$. Assume
$\mathscr{N}=\{S'_{j_{1}},S'_{j_{2}}\}$ and $|S'_{j_{1}}\cap X|=1$.
We know that there is some $S'_{j}\subseteq
V(G)\backslash(V(H_{1})\cup S'_{j_{1}}\cup S'_{j_{2}})$. By Lemma
\ref{5rec}(a),
\begin{align*}
|S|=\lambda_{s}(G)&\leq d_{G}(V(H_{1})\cup S'_{j_{1}}\cup S'_{j_{2}})\\
&\leq d_{G}(H_{1})-d_{G[S'_{j_{2}}]}(S'_{j_{2}}\cap V(H_{1}))+|S'_{j_{2}}\backslash V(H_{1})|-(k-2)\\
&=d_{G}(H_{1})+2|E(S'_{j_{2}}\cap X)|-2(k-1)+|S'_{j_{2}}\backslash V(H_{1})|-(k-2)\\
&\leq3k-3k+4+|S'_{j_{2}}\backslash V(H_{1})|.
\end{align*}
Similarly, we can obtain $|S'|\leq d_{G}((V(H_{1})\cup
S'_{j_{1}})\backslash S'_{j_{2}})\leq4+|S'_{j_{2}}\cap V(H_{1})|$.
Then $2|S'|\leq 8+|S'_{j_{2}}|$, which implies $|S'|\leq8<2k$, a
contradiction.

The proof is complete.


\begin{thebibliography}{6}
\addtolength{\itemsep}{-2ex}
\bibitem{Andrasfai} B. Andr\'{a}sfai, P. Erd\"{o}s, V. T. S\'{o}s,
On the connection between chromatic number, maximal clique and
minimal degree of a graph, Discrete Math. 8 (1974) 205-218.
\bibitem{Biggs} N. Biggs, Algebraic Graph Theory, Cambridge University Press, Cambridge, 1993, pp. 36-38.
\bibitem{Chan} O. Chan, C.C. Chen, Q. Yu, On 2-extendable obelian Cayley graphs, Discrete Math. 146 (1995) 19-32.
\bibitem{Chen} C.C. Chen, J. Liu, Q. Yu, On the classification of 2-extendable Cayley graphs on dihedral groups, Australas. J. Combin. 6 (1992) 209-219.
\bibitem{Diestel} R. Diestel, Graph Theory, Springer, New York,
2006, p. 41.
\bibitem{Fabrega} J. F\`{a}brega, M.A. Fiol, Extraconnectivity of graphs with large girth, Discrete Math. 127 (1994) 163-170.
\bibitem{Favarvon} O. Favaron, On $k$-factor-critical graphs, Discuss. Math. Graph Theory 16 (1996) 41-51.
\bibitem{O. Favaron} O. Favaron, Extendability and factor-criticality, Discrete Math. 213 (2000) 115-122.
\bibitem{Gallai} T. Gallai, Neuer Beweis eines Tutte-schen Satzes, Magyar Tud. Akad. Matem. Kut. Int. K\"{o}zl. 8 (1963) 135-139.
\bibitem{Heuvel} J. Heuvel, B. Jacskon, On the edge connectivity, hamiltonicity, and
toughness of vertex transitive graphs, J. Combin. Theory Ser. B 77
(1999) 138-149.
\bibitem{Holtkamp} A. Holtkamp, D. Meierling, L.P. Montejano, $k$-restricted edge-connectivity in triangle-free graphs, Discrete Appl. Math. 160 (2012) 1345-1355.
\bibitem{lovasz} L. Lov\'{a}sz, On the structure of factorizable graphs, Acta Math. Acad. Sci. Hungar. 23 (1972) 179-195.
\bibitem{plummer} L. Lov\'{a}sz, M.D. Plummer, Matching Theory, North-Holland, Amsterdam, 1986, p. 207.
\bibitem{Mader} W. Mader, Minimale $n$-fach kantenzusammenh\"{a}ngenden Graphen, Math. Ann. 191 (1971) 21-28.
\bibitem{Mantel} W. Mantel, Problem 28, Wiskundige Opgaven 10 (1907)
60-61.
\bibitem{Miklavi} \v{S}. Miklavi\v{c}, P. \v{S}parl, On extendability of Cayley graphs, Filomat 23 (2009) 93-101.
\bibitem{M.D. Plummer} M.D. Plummer, On $n$-extendable graphs, Discrete Math. 31 (1980) 201-210.
\bibitem{R. Tindell}R. Tindell,  Connectivity of Cayley graphs, in: D.Z. Du, D.F. Hsu (Eds.). Combinatorial Network Theory, Kluwer, Dordrecht, 1996, pp. 41-64.
\bibitem{Watkins} M.E. Watkins, Connectivity of transitive graphs, J. Combin. Theory 8 (1970) 23-29.
\bibitem{Yu} Q. Yu, Characterizations of various matching extensions in graphs, Australas. J. Combin. 7 (1993) 55-64.
\bibitem{Yang} M. Yang, Z. Zhang, C. Qin, X. Guo, On super 2-restricted and 3-restricted edge-connected vertex transitive graphs, Discrete Math. 311 (2011) 2683-2689.
\bibitem{Zhang} H. Zhang, W. Sun, 3-Factor-criticality of vertex-transitive graphs, J. Graph Theory, revised, arXiv: 1212.3940.
\end{thebibliography}
\end{document}